\documentclass[a4paper,abstracton]{scrartcl}

\def\arxivVersion{1}

\usepackage[a4paper,top=1in,bottom=1in,left=1.32in,right=1.32in]{geometry}

\usepackage[utf8]{inputenc}

\usepackage{caption}
\usepackage{amssymb, amsfonts, amsthm}
\usepackage{amsmath}
\usepackage{hyperref}
\usepackage{cleveref} 
\usepackage{multirow}
\usepackage{array} 
\usepackage{footnote}
\usepackage{booktabs}
\usepackage{mathtools}
\usepackage{siunitx}

\usepackage{algorithm}
\usepackage{algpseudocode}
\usepackage{varwidth}

\usepackage{tablefootnote}
\usepackage{graphicx}
\usepackage{adjustbox}
\usepackage[english]{babel}

\usepackage{import}

\usepackage[normalem]{ulem}

\usepackage{pgfplots}
\usepackage{tikz}
\usetikzlibrary{shapes.geometric, arrows.meta, positioning}

\newtheorem{theorem}{Theorem}
\newtheorem{lemma}[theorem]{Lemma}
\newtheorem{proposition}[theorem]{Proposition}

\newtheorem{definition}[theorem]{Definition}
\newtheorem{assumption}[theorem]{Assumption}

\newcommand{\R}{\ensuremath{\mathbb{R}}}
\newcommand{\N}{\ensuremath{\mathbb{N}}}

\DeclareMathOperator{\argmax}{argmax}
\DeclareMathOperator{\MAS}{MAS}
\newcommand{\Q}[1]{\min_{y^s \in \mathcal{Y}^s(#1)} g^s(#1, y^s)}

\newcommand{\CCG}{C\&CG}
\newcommand{\OPT}{\ensuremath{\textnormal{OPT}}}
\newcommand{\val}{\textnormal{VAL}}
\newcommand{\UB}{\textnormal{UB}}
\newcommand{\LB}{\textnormal{LB}}
\newcommand{\RO}{\textnormal{2-RO}}
\newcommand{\FBS}{\texttt{FindBadScenario}}
\newcommand{\Init}{\texttt{InitSubset}}

\newcommand{\rev}[1]{\textcolor{black}{#1}}
\newcommand{\revv}[1]{\textcolor{black}{#1}}

\AtBeginEnvironment{align}{\setcounter{subeqn}{0}}
\newcounter{subeqn} \renewcommand{\thesubeqn}{\theequation\alph{subeqn}}%
\newcommand{\subeqn}{%
    \refstepcounter{subeqn}
    \tag{\thesubeqn}
}

\usepackage{authblk}

\begin{document}


 \title{A fast approximate column-and-constraint generation method for two-stage robust mixed-integer programs}

\author[1]{Marc Goerigk}
\author[1]{Dorothee Henke}
\author[2]{Johannes Kager\footnote{Corresponding author. Email: johannes.kager@tum.de}}
\author[3]{Fabian Schäfer}
\author[2,4]{Clemens Thielen}

\date{}

\affil[1]{Business Decisions and Data Science, University of Passau, Germany}
\affil[2]{Professorship of Optimization and Sustainable Decision Making, Campus Straubing for Biotechnology and Sustainability, Technical~University~of~Munich, Germany}
\affil[3]{Chair of Supply and Value Chain Management, Campus Straubing for Biotechnology and Sustainability, Technical~University~of~Munich, Germany}
\affil[4]{Department of Mathematics, School of Computation, Information and Technology, Technical University of Munich, Germany}

\maketitle

\begin{abstract}
This paper presents a new column-and-constraint generation method for two-stage robust mixed-integer programs with finite uncertainty sets. Our method combines and extends speed-up techniques used in previous column-and-constraint generation methods and introduces several new techniques. In particular, it uses dual bounds for second-stage problems in order to allow a faster identification of the next promising scenario to be added to the master problem.
Moreover, adaptive time limits are imposed to avoid getting stuck on particularly hard second-stage problems, and a gap propagation between master problem and second-stage problems is used to stop solving them earlier if only a given non-zero optimality gap is to be reached overall. 
This makes our method particularly effective for problems where solving the second-stage problem is computationally challenging.
To evaluate the method's performance, we compare it to two recent column-and-constraint generation methods from the
literature on two applications: a robust capacitated location routing problem and a robust integrated berth allocation and quay crane assignment and scheduling problem. The first problem features a particularly hard second stage, and we show that our method is able to solve considerably more and larger instances in a given time limit.
Using the second problem, we verify the general applicability of our method, even for problems where the second stage is relatively easy.
\end{abstract}

\noindent\textsf{\textbf{Keywords:}}
Robust optimization; Column-and-constraint generation; Finite uncertainty set; Two-stage robust optimization

\section{Introduction}\label{sec:intro}

Robust optimization \cite{ben2009robust} aims to find solutions of optimization problems that are uncertain with respect to realizations of some problem parameters within a given uncertainty set.
In \emph{two-stage robust optimization problems}, some decisions \emph{(here-and-now decisions)} have to be made before information about the uncertain parameters is revealed, but other decisions \emph{(wait-and-see decisions)} can be postponed until some or even all information is available \cite{YANIKOGLU2019799}. In this work, we consider two-stage robust mixed-integer programs with uncertainty sets consisting of finitely many scenarios as studied, e.g., in~\cite{bib29bienstock2008computing}, \cite{toenissen11}, \cite{RODRIGUES2021499} and \cite{bib18chargui2023berth}.

Let us denote this finite set of scenarios by~$S$.
The here-and-now decisions are made in the first stage of the optimization and are represented by a vector~$x$ of first-stage variables.
The wait-and-see decisions are made in the second stage of the problem
and are represented by vectors~$y^s$ of second-stage variables for all scenarios~$s\in S$, which are summarized in the vector~$y=(y^s)_{s\in S}$ in an arbitrary order.
We let~$\mathcal X$ denote the feasible set of the first-stage variables~$x$ and, for each~$x \in \mathcal X$ and~$s \in S$, let~$\mathcal{Y}^s(x)$ denote the feasible set of the second-stage variables~$y^s$. We assume that these sets can be expressed as feasible sets of mixed-integer programs~(MIPs), i.e., using finitely many linear equality and inequality constraints and integrality constraints on some or all of the variables. Further, $f(x)$ describes a non-negative affine objective function of the first stage, and, for each~$s \in S$, $g^s(x, y^s)$ describes a non-negative affine objective function (in~$x$ and~$y^s$) of the second stage.

Using this notation, the two-stage robust mixed-integer programs we consider can be written as follows:
\begin{align}\label{eq:mod}\tag{\RO}
	\min_{x \in \mathcal{X}} \left( f(x) + \max_{s \in S} \min_{y^s\in\mathcal{Y}^s(x)} g^s(x, y^s) \right)
\end{align}

For a fixed first-stage solution~$x \in \mathcal{X}$ and a scenario $s \in S$, we call the inner minimization problem the \emph{second-stage problem} and refer to its value $Q(x, s) \coloneqq \Q{x}$ as the \emph{second-stage cost of~$x$ in scenario~$s$}, or simply as the \emph{(second-stage) cost of scenario~$s$} if~$x$ is clear from the context. With slight abuse of notation, we also refer to $Q(x, s)$ as the second-stage problem.
Upper and lower bounds on the second-stage cost of~$x$ in scenario~$s$ are denoted by~$\UB^s(x) \in \R_{\ge 0} \cup \{+\infty\}$ and~$\LB^s(x) \in \R_{\ge 0}$, respectively. Here, we shortly write $\UB^s\coloneqq \UB^s(x)$ and $\LB^s\coloneqq \LB^s(x)$ when the first-stage solution~$x$ is clear from the context, and we set $\UB^s(x)\coloneqq +\infty$
if the second-stage problem~$Q(x,s)$ is infeasible. The cost or the upper bound of a scenario is \emph{worse} than the one of another scenario when it is strictly larger.
Given a first-stage solution~$x\in\mathcal{X}$ and a subset~$R\subseteq S$ of the scenarios, we say that a scenario~$s \in R$ is \emph{worst for~$x$ in~$R$} if $s \in \argmax_{k\in R} Q(x,k)$, and \emph{worst for~$x$ in~$R$ with respect to~$\UB$} if $s \in \argmax_{k\in R} \UB^k(x)$. A \emph{worst scenario for~$x$ (with respect to~$\UB$)} is one that is worst for~$x$ (with respect to~$\UB$) in the whole set~$S$ of scenarios.

\medskip

Since there is a finite number of scenarios, we can transform the min-max-min-problem~\eqref{eq:mod} into an equivalent single-stage minimization problem. Below, we perform this reformulation in the more general case where~$S$ is replaced by any subset~$D\subseteq S$ of scenarios. For~$D\subseteq S$, the resulting MIP is called the \emph{master problem}~$\MAS_D$:
\begin{equation*} 
\begin{aligned}
    \quad \min \quad&f(x) + z \\
    \mathrm{s.t.} \quad&z \ge g^s(x, y^s) &&
    \forall s \in D\\
    &z \in \R_{\ge 0}, \;\;x \in \mathcal{X},\;\; y^s\in\mathcal{Y}^s(x) && \forall s \in D
\end{aligned}
\end{equation*}

In~$\MAS_D$, the new auxiliary variable~$z$ upper-bounds the second-stage objective for all scenarios in~$D$, which means that minimizing~$z$ is equivalent to minimizing the maximum second-stage cost over all scenarios in~$D$. In particular, for~$D=S$, the resulting mixed-integer program~$\MAS_S$ is equivalent to~\eqref{eq:mod}.
A feasible solution of~$\MAS_D$ is a triple~$(x,y,z)$.
However, when speaking about a~\emph{master solution}, we mean only the pair~$(x, z)$ since these are the variables that are fixed when focusing on the second-stage after solving~$\MAS_D$.

The consideration of subsets~$D\subseteq S$ of scenarios in~$\MAS_D$ is motivated by the observation that solving the
master problem for the whole set~$S$ of scenarios is often computationally intractable for problems of a practically relevant size.
This is especially the case when the second-stage problems are hard to solve, or when the number of scenarios is large. Thus, based on the fact that only a subset of scenarios might actually be relevant for finding an optimal first-stage solution, the idea of~\rev{\emph{column-and-constraint generation (\CCG{}) methods}} is to start with a small number of scenarios and iteratively add a worst scenario for the current first-stage solution \cite{zeng}.
\rev{The procedure terminates when a \revv{termination condition} certifies that adding the remaining scenarios in~$S \setminus D$ cannot deteriorate the current solution's objective value by more than a desired tolerance.}
Thus, the general structure of a column-and-constraint generation method can be broadly summarized as follows:
\begin{enumerate}
    \item Let $D \subseteq S$ be a small set of scenarios (usually one scenario only).
    \item Solve~$\MAS_D$ and denote the master solution by~$(\tilde x, \tilde z)$.\label{it:ccgstep2}
    \item Find a worst scenario~$s \in S$ for the given first-stage solution~$\tilde x$.\label{it:ccgstep3}
    \item \rev{Unless a} \revv{termination condition} \rev{is met, add~$s$ to~$D$ and go to Step~\ref{it:ccgstep2}.}
\end{enumerate}

The effectiveness of this method highly depends on the effort of finding a worst scenario for the current first-stage solution (Step~\ref{it:ccgstep3}).
We call this problem of finding a worst scenario for the current first-stage solution in each iteration the \emph{subproblem} (of finding a worst scenario).
A central possibility for improving the method is to apply heuristics in the process of finding a worst scenario, which removes the necessity to solve the second-stage problem of each scenario to optimality.
\rev{In fact, since the number of scenarios is finite, termination of a \CCG{} method is guaranteed whenever solving the subproblem always leads to either the \revv{termination condition} being met or a new scenario~$s \in S \setminus D$ being returned.}

\section{Related literature and our contribution}

Two-stage robust programming problems were first
\rev{introduced} in~\cite{bib27ben2004adjustable}. In comparison to conventional robust programs, the authors allow the values of certain variables (the second-stage variables) to be determined after the realization of some uncertain parameters. In contrast, the values of the remaining variables (the first-stage variables) must be determined before the uncertain parameters are realized and cannot be changed afterwards. \rev{A similar concept is the one of recoverable robust optimization. Here, a solution to the problem is fully described by the first-stage variables, and it can be modified (with penalty costs) in the second stage \cite{bib4liebchen2009concept,bib6goerigk2014recovery}.}

\rev{It is important to note that, unlike in single-stage min–max problems, where a worst-case scenario in a convex uncertainty set lies at an extreme point of the set, the situation is different in two-stage problems. Here, a discrete scenario set cannot simply be replaced by its convex hull, and polyhedral reformulation techniques are not applicable (see also Section~4.4.2 in \cite{goerigk2024buch}). This further emphasizes the importance of specialized algorithms for finite uncertainty sets}%
\revv{, and several techniques have been developed in the literature for two-stage robust optimization problems in this setting.} As discussed in the introduction, one approach is to reformulate them as single-stage problems. However\revv{, especially for two-stage robust mixed-integer programs,} the resulting single-stage problems quickly become intractable, motivating the development of decomposition algorithms such as \CCG{} methods.

\CCG{} methods can be seen as an \emph{adversarial} approach \cite{luceraiffabook}, i.e., two players are involved where one player proposes values for the first-stage variables using only a subset of scenarios, denoted by~$D$ in~$\MAS_D$, and the other player answers by delivering an additional scenario for which the proposed values of the first-stage variables are not optimal \revv{(up to a desired tolerance)}. If no such additional scenario is found, the algorithm can terminate.

One of the first versions of a \CCG{} method was published in 2008 in~\cite{bib29bienstock2008computing} under the name of \emph{approximate adversarial algorithm}. The authors use it for computing robust basestock levels. The method also appears in the survey \cite{aissi2009min} in the context of min-max regret problems. \rev{The name \emph{column-and-constraint generation method} was first coined in \cite{zeng}, where it was formalized as a general method for two-stage problems. The authors apply the algorithm to a two-stage robust location-transportation problem. Another name for \CCG{} methods that appears in the literature is \emph{scenario addition methods}~\cite{toenissen11,toenissen12}.}

\medskip

The main advantage of \CCG{} methods is that they reduce the complexity with respect to the number of scenarios. In many applications, it can be observed that, when iteratively identifying a bad or even worst scenario for the current first-stage solution, only a handful of scenarios need to be added to the master problem to solve the problem to optimality, even when the number of scenarios is in the hundreds. While this allows \CCG{} methods to reduce the size of the considered master problems significantly, solving the second-stage problem for every scenario in each iteration to find a worst scenario can be quite time-consuming in cases where the second-stage problem is hard. Therefore, several methods have been developed that aim at identifying a bad or worst scenario for the current first-stage solution without solving the second-stage problem optimally for every scenario.

The \emph{improved scenario addition method} (ISAM) proposed in~\cite{toenissen11} aims at accelerating the search for a worst scenario for the current first-stage solution by first applying a heuristic to the second-stage problem of each scenario. The second-stage problems are then solved to optimality in order of non-increasing heuristic objective values. This procedure may find a worst scenario before solving all second-stage problems to optimality if the optimal objective value of a scenario's second-stage problem is found to be
larger than or equal to the (heuristic) second-stage costs of all other scenarios. To evaluate their algorithm, the authors of~\cite{toenissen11} use a recoverable robust maintenance location routing problem for rolling stock. Since the ISAM from~\cite{toenissen11} is well-suited for applications with hard second-stage problems, we use it to compare our new \CCG{} method to and provide a more formal description of the algorithm in Section~\ref{sec:alg-previous}.

Another recent method aimed at quickly finding a bad scenario for the current first-stage solution is the \emph{scenario reduction procedure} (SRP) proposed in~\cite{RODRIGUES2021499}. The main new idea is that a scenario does not have to be added to~$D$ if the best upper bound for its second-stage cost is smaller than or equal to the value~$\tilde z$ obtained from the current master problem, and the algorithm can terminate if this is the case for all scenarios. Therefore, the proposed algorithm goes through the list of scenarios in order to determine if any of them have a second-stage cost strictly larger than~$\tilde z$. This is done by first applying a fast heuristic to the second-stage problem of a scenario and only solving the corresponding second-stage problem optimally afterwards if the heuristic value is strictly larger than~$\tilde z$. If the resulting optimal second-stage cost is still strictly larger than~$\tilde z$, the corresponding scenario is directly added to~$D$ without considering any further scenarios.
Thus, the algorithm does not necessarily find a worst scenario for the current first-stage solution, but instead returns the first considered scenario for which the second-stage cost of the current first-stage solution is larger than~$\tilde z$, i.e., larger than the current first-stage solution's second-stage cost in all scenarios already contained in~$D$. The algorithm is applied to an integrated berth allocation and quay crane assignment and scheduling problem (BACASP) where arrival times are uncertain. We use it  to compare our new \CCG{} method and provide a more detailed description in Section~\ref{sec:alg-previous}.

A slightly different version of the SRP is used in~\cite{bib18chargui2023berth} for a BACASP with renewable energy uncertainty. In the subproblem, all second-stage problems of scenarios whose heuristic value is larger than~$\tilde z$ are solved to optimality in order to find a worst scenario. The used heuristic is capable of finding a feasible solution to both the master and the second-stage problems and is based on a variable and iterated local neighborhood search algorithm.

Recently, authors tried to combine two decomposition methods for two-stage robust problems \cite{bib16li2024decomposition}. The algorithm alternates between adding cuts obtained from a Benders decomposition and from a column-and-constraint generation method to the master problem.

The authors of~\cite{bib17wang2024robust} propose a so-called scenario-constrained \CCG{} method. The worst scenarios that are identified in each iteration are added to a list, but only the most recently identified scenario is used for the master problem. If a scenario re-occurs as a worst scenario, the algorithm can be terminated.

The case where the hardness of the second-stage problem is not the limiting factor but the master problem is hard to solve even for a small number of scenarios is considered in~\cite{TSANG202392}. After solving the master problem up to a certain gap, they solve the second-stage problem optimally for every scenario. Afterwards, they only add a worst scenario if this yields a sufficiently large increase in the second-stage cost of the current first-stage solution. Otherwise, they continue to solve the master problem to a smaller gap first before solving the second-stage problems again for the improved first-stage solution.

Another, more sophisticated method is described in~\cite{hashemi}. It augments a Benders decomposition algorithm with a cut-generating heuristic that results from a reformulation of the two-stage robust problem into a single-stage minimization problem. This is done by pulling out the inner minimization problem, replacing the finite uncertainty set by its convex hull, and dualizing the remaining maximization part. The authors demonstrate the effectiveness of the algorithm on a two-stage nurse planning and a two-echelon supply chain problem.

\medskip

Finally, we also note that, most recently, machine learning techniques have been applied to speed up the solution process for two-stage robust optimization problems, see \cite{bib13bertsimas2024machine} and \cite{
goerigk2024data}.

\paragraph{\textbf{Our contribution}}
In this work, we propose a new \CCG{} method and evaluate its performance on two applications.

The most important improvements used in our algorithm can be summarized as follows (see Section~\ref{sec:alg} for details). First, we consider lower (dual) bounds for the second-stage problems while solving the subproblem. After applying a quick heuristic, we pursue a top-to-bottom strategy similar to~\cite{toenissen11}, but
\rev{we always solve a second-stage problem attaining the largest current upper bound and switch to another problem whenever the ordering changes.} Moreover, we exclude a scenario from further consideration as soon as the current upper bound for its second-stage problem becomes less than or equal to the lower bound of another scenario's second-stage problem.
Second, we use an adaptive time limit for each scenario's second-stage problem when
\rev{searching for a worst scenario}. This avoids getting stuck on a particularly hard second-stage problem and can therefore reduce the overall runtime significantly as shown in our computational experiments. The time limit is adaptive in the sense that it depends linearly on the time spent for solving the master problem in the current iteration.
Third, we allow passing a non-zero, user-defined \emph{target gap} up to which~\eqref{eq:mod} is to be solved to our \CCG{} method. \rev{The gap used when solving the master problem in each iteration is set relative to this target gap, and we are able to increase the upper bound~$\tilde z$ on the worst second-stage cost of the current master solution in cases where this master solution has a smaller gap than the target gap}.

We prove the correctness and termination of our method, which we call \emph{approximate scenario bracketing procedure} (ASBP), and compare it to the state-of-the-art \CCG{} methods from~\cite{toenissen11} (\emph{improved scenario addition method} (ISAM)) and~\cite{RODRIGUES2021499} (\emph{scenario reduction procedure} (SRP)), which are described in Section~\ref{sec:alg-previous}. \rev{A feature comparison of our method and these methods is provided in Table~\ref{tab:algorithm-comparison}}. We compare the three methods on two applications described in Section~\ref{sec:applications}.
The first application is a robust version of a capacitated location routing problem \cite{alvarez,toro} with hard second-stage problems (i.e., more time is spent on the second-stage problems than on the master problems overall). Here, our method outperforms both the ISAM and the SRP considerably even for a target gap of zero, and the margin of improvement is even larger when a non-zero target gap is used. The second application is a robust integrated berth allocation and quay crane assignment and scheduling problem \cite{AGRA2018138,RODRIGUES2021499} with comparatively easy second-stage problems (i.e., less time is spent on the second-stage problems than on the master problems overall). Even though our method is mostly targeted at problems with hard second-stage problems, the computational results show that it still outperforms the ISAM by a notable margin for this problem, and it \revv{even shows a better performance than the SRP which has been specifically designed for this application}.

Moreover, we perform an analysis of the impact of the different speed-up techniques employed in our method by comparing it to several variants in which one of these techniques is disabled in each case.

\begin{table}[H]
  \footnotesize
  \centering
  \begin{tabular}{lccc}
    \toprule
    \textbf{Feature} & \textbf{ASBP} & \textbf{ISAM} & \textbf{SRP} \\
    & (this work) &
    &
    \\
    \midrule
    Lower (dual) bounds for second-stage problems & \checkmark & \texttimes & \texttimes \\ \addlinespace
    Adaptive time limits for second-stage problems & \checkmark & \texttimes & \texttimes \\ \addlinespace
    Gap propagation between master and & \checkmark & \texttimes & \texttimes \\
    second-stage problems & & & \\ \addlinespace
    %
    \revv{Always solves second-stage problem with}
                     & \checkmark & \texttimes & \texttimes \\
    \revv{largest current upper bound} & & & \\  \addlinespace
    Considers $\tilde z$-bound from master problem in & \checkmark & \texttimes & \checkmark \\
    subproblem & & & \\ \addlinespace
    Subproblem returns a worst scenario & (\checkmark)\tablefootnote{
                                          Decided adaptively based on time limit for second-stage problems.
                                          } & \checkmark & \texttimes \\ \addlinespace
    \bottomrule
  \end{tabular}
  \caption{Comparison of three \CCG{} methods: our ASBP, the ISAM from \cite{toenissen11}, and the SRP from \cite{RODRIGUES2021499}.}
  \label{tab:algorithm-comparison}
\end{table}

\section{\rev{General structure and known \CCG{} methods}}\label{sec:alg-previous}

This section formalizes the general structure of a \CCG{} method as well as the specific state-of-the-art \CCG{} methods from \cite{toenissen11} and \cite{RODRIGUES2021499} that we compare our new method to.

\medskip

The general structure of a \CCG{} method for a mixed-integer program (MIP) whose objective is to be minimized is shown in Algorithm~\ref{alg:general}. To obtain a complete description of a specific column-and-constraint generation method, the implementations of the two subroutines \Init() and \FBS()
need to be specified, and several variants for this will be discussed below. Moreover, Algorithm~\ref{alg:general} allows to input a target gap~$P\in [0,1)$ up to which~\eqref{eq:mod} is to be solved, \rev{and a \emph{master-gap factor}~$\mu$ used to determine to which gap $\MAS_D$ is to be solved}. These will be discussed in detail in Section~\ref{sec:alg-gaps}.

\begin{algorithm}
	\caption{General structure of a \CCG{} method}\label{alg:general}
	\begin{algorithmic}[1]
 \Statex \textbf{Parameter:} Target gap $P \in [0,1)$, \rev{master-gap factor $\mu \in [0,1]$}
		\Statex \textbf{Input:} Scenario set~$S$
		\Statex \textbf{Output:} A first-stage solution of~\eqref{eq:mod} with gap at most~$P$ \rev{or the information that the problem is infeasible}
		\Statex \textbf{Initialization:} $D \gets $ \Init().
            \State $\text{Terminate} \gets \text{False}$.
		\While{$\text{Terminate} = \text{False}$}\label{alg:general-while-loop}
                \State \rev{Try to solve~$\MAS_D$ with gap~$\mu P$.} \label{algline:1step1}
                \rev{
                \If{$\MAS_D$ is infeasible}
			     \State \Return Problem infeasible.
                 \Else
                 \State Let $(\tilde{x}, \tilde{z})$ denote the obtained master solution.
                \EndIf
                }
                \State $s \gets \FBS(D, \tilde{x}, \tilde{z})$.   \label{algline:general:subroutine}
		      \If{$s \in D$} \label{algline:generalterminationcrit}
			     \State $\text{Terminate} \gets \text{True}$.
		      \Else
			     \State Add~$s$ to~$D$.
                \EndIf
		\EndWhile \label{alg:general-while-loop-end}
            \State \Return $\tilde x$.
	\end{algorithmic}
\end{algorithm}

\medskip

The implementation of \Init() can be chosen in different ways in each of the considered \CCG{} methods. Possible options include:
\begin{enumerate}
    \item[1)] $\Init()=\emptyset$. In the first iteration of Lines~\ref{alg:general-while-loop}--\ref{alg:general-while-loop-end} in Algorithm~\ref{alg:general}, the master problem is then~$\MAS_\emptyset$ and the variable~$z$ is only bounded from below by zero. Therefore, only~$f(x)$ is minimized when solving~$\MAS_{\emptyset}$ in Line~\ref{algline:1step1} of the first iteration, and \rev{if~$\MAS_{\emptyset}$ is feasible,} the first scenario added to~$D$ will be the one found by \FBS() in Line~\ref{algline:general:subroutine}.
    \item[2)] $\Init()=\{s\}$ for a scenario~$s\in S$.
    \item[3)] $\Init()=M$, where~$M$ denotes a set of two or more scenarios.
\end{enumerate}

\rev{Note that, in the second (third) option, the set can be initialized with one (several) randomly chosen scenario(s).
Alternatively, scenarios that are expected to yield a high second-stage cost, identified using some problem-specific procedure, are often used.}

In the following description of the \CCG{} methods from \cite{toenissen11} and \cite{RODRIGUES2021499}, we state the implementation used in the corresponding paper for each method.
\rev{However, all of the three presented options would work in each of the \CCG{} methods. In fact, }in our computational experiments in Section~\ref{sec:computational-exps}, we choose the implementation of \Init() uniformly across the different compared methods for each considered application in order to obtain a fair comparison.

Further, we note that the \CCG{} methods from \cite{toenissen11} and \cite{RODRIGUES2021499} do not use any target gap in their original version, which corresponds to a target gap of~$P=0$ in Algorithm~\ref{alg:general}. Thus, they always solve the problem optimally. In our computational comparisons in Section~\ref{sec:computational-exps}, however, we apply all algorithms with target gap zero as well as with non-zero target gaps. Therefore, we provide a short proof of correctness for both algorithms in the case of a non-zero target gap in %
\if\arxivVersion0%
the supplementary material (Section~A).
\else%
\ref{sec:appendix-1-proofs}.
\fi

\medskip

The \emph{improved scenario addition method} (ISAM) presented in~\cite{toenissen11} uses $\Init()=\{s\}$ for a random scenario~$s\in S$, and its implementation of \FBS() is shown in Algorithm~\ref{alg:toenissen}.
The main idea of this implementation of \FBS() is to first apply a fast heuristic for the second-stage problem of each scenario in order to obtain upper bounds~$\UB^s$ for all~$s \in S$ quickly. Afterwards, it iteratively solves the second-stage problem for a scenario~$s$ with maximum upper bound~$\UB^s$ to optimality and updates the corresponding upper bound to the obtained optimal objective value. The subroutine stops and returns the scenario~$k$ corresponding to the last optimally-solved second-stage problem if this problem's optimal objective value is larger or equal to all current upper bounds~$\UB^s$, in which case~$k$ must be a worst scenario for the first-stage solution~$\tilde{x}$ that was provided as an input to \FBS().

 \begin{algorithm}
    \caption{Implementation of $\FBS()$ in the ISAM \cite{toenissen11}}\label{alg:toenissen}
 	\begin{algorithmic}[1]
 		\Statex \textbf{Input:} Scenario set~$D \subseteq S$ and master solution~$(\tilde x, \tilde z)$ of~$\MAS_D$
 		\Statex \textbf{Output:} A scenario~$s \in S$
 		\Statex \textbf{Initialization:} $\UB^s \gets +\infty \;\forall s\in S$.
 		\For{$s \in S$}
 		\State \parbox[t]{\dimexpr\linewidth-\algorithmicindent}{
 		Run a heuristic for~$Q(\tilde x, s)$. If the heuristic finds a feasible solution, set~$\UB^s$ to the corresponding second-stage cost.
 		}
 		\EndFor
 		\State Let $k \in \argmax_{s\in S} \UB^s$. \label{algline:2step5}
 		\State Optimally solve $Q(\tilde x, k)$ and update~$\UB^k$.
 		\If{$k \notin \argmax_{s \in S} \UB^s$}
 		\State Go to Line~\ref{algline:2step5}.
 		\EndIf
 		\State \Return $k$. \label{algline:isamreturn}
 	\end{algorithmic}
 \end{algorithm}

\rev{The \emph{scenario reduction procedure} (SRP) from \cite{RODRIGUES2021499} uses \linebreak$\Init()=\{s\}$ with~$s\in S$,\footnote{Note that, in the experiments conducted in~\cite{RODRIGUES2021499}, $s$ either denotes the nominal scenario or a scenario that is expected to yield a high second-stage cost and is identified using a problem-specific procedure. In the latter case, the authors of~\cite{RODRIGUES2021499} speak of a ``warm start'' and correspondingly denote the algorithm as \emph{SRP+WS}.} and its implementation of \FBS() is shown in Algorithm~\ref{alg:rodriguez}.}
Here, the main new idea is that adding a scenario~$s\in S\setminus D$ to~$D$ only increases the objective value $f(\tilde{x})+\tilde{z}$ of the provided master solution~$(\tilde{x},\tilde{z})$ in~$\MAS_D$ if the optimal objective value of the second-stage problem~$Q(\tilde{x},s)$ is larger than~$\tilde{z}$. Therefore, Algorithm~\ref{alg:rodriguez} stops the heuristic for a scenario's second-stage problem when the objective value becomes less than or equal to~$\tilde{z}$. Moreover, the algorithm does not necessarily apply the heuristic for all scenarios. Instead, it only applies it for the scenarios in~$S\setminus D$ and, whenever the heuristic's objective value for a scenario's second-stage problem is larger than~$\tilde{z}$, this second-stage problem is immediately solved to optimality using an MIP solver. If the obtained optimal objective value still exceeds~$\tilde{z}$, Algorithm~\ref{alg:rodriguez} immediately returns the corresponding scenario without considering any of the remaining scenarios. In case no scenario in~$S\setminus D$ has an optimal second-stage objective value larger than~$\tilde z$, the overall algorithm (Algorithm~\ref{alg:general}) can be terminated since a worst scenario for the current first-stage solution~$\tilde{x}$ is already contained in~$D$. In Algorithm~\ref{alg:rodriguez}, this is modeled by returning an arbitrary scenario~$s \in D$, as this leads to termination of Algorithm~\ref{alg:general}.
Overall, if Algorithm~\ref{alg:rodriguez} returns a scenario in~$S\setminus D$, this scenario is necessarily worse for the current first-stage solution~$\tilde{x}$ than all scenarios previously contained in~$D$, but it is not necessarily a worst scenario for~$\tilde{x}$.

 \begin{algorithm}
    \caption{Implementation of $\FBS()$ in the SRP \cite{RODRIGUES2021499}}\label{alg:rodriguez}
	\begin{algorithmic}[1]
		\Statex \textbf{Input:} Scenario set~$D \subseteq S$ and master solution~$(\tilde x, \tilde z)$ of~$\MAS_D$
		\Statex \textbf{Output:} A scenario~$s \in S$
		\Statex \textbf{Initialization:} $\UB^s \gets +\infty \;\forall s\in S$.
		\For{$s \in S \setminus D$}
		\State \parbox[t]{\dimexpr\linewidth-\algorithmicindent}{
		Run a heuristic for~$Q(\tilde x, s)$. Stop the heuristic when the objective value becomes~$\leq \tilde z$. If the heuristic finds a feasible solution, set~$\UB^s$ to the corresponding second-stage cost.
		} \label{ouralgline:heuristic}
		\If{$\UB^s > \tilde z$}
			\State Optimally solve $Q(\tilde x, s)$ and update~$\UB^s$.
			\If{$\UB^s > \tilde z$}
				\State \Return $s$.
			\EndIf
		\EndIf
		\EndFor
		\State \Return any $s \in D$. \label{algline:srpreturn}
	\end{algorithmic}
\end{algorithm}

\section{The approximate scenario bracketing procedure} \label{sec:alg}

In Section~\ref{sec:alg-gaps}, we first prove two results that form the basis for exploiting a non-zero target gap for~\eqref{eq:mod} in our new approximate \CCG{} method\rev{, called the \emph{approximate scenario bracketing procedure} (ASBP). We present our method in Section~\ref{sec:alg-ours}, and its correctness is shown in Section~\ref{sec:alg-ours-analysis}.}
\subsection{Gap propagation in column-and-constraint generation methods} \label{sec:alg-gaps}

We first give a precise definition of the gap in the context of optimization problems.

\begin{definition}\label{def:gap}
A feasible solution with objective value~$\val$ for a minimization problem with optimal objective value~$\OPT\ge 0$ has \emph{gap}~$p\in[0,1)$ if
\begin{align*}
    (1-p)\cdot\val \le \OPT \text{ or, equivalently, if } \val \le \frac{1}{1-p}\cdot\OPT.
\end{align*}
\end{definition}

Note that this definition, in particular, implies that a feasible solution is optimal if and only if it has gap zero.

Besides an upper (primal) bound~$\UB$ given by the objective value of the currently best solution, mixed-integer programming solvers  usually provide also a lower (dual) bound~$\LB$ on the optimal objective value~$\OPT$ of the considered problem at each point in the solution process. Using these two bounds, the \emph{current gap} is then defined as $p = \frac{\UB - \LB}{\UB}$ (see, e.g., \cite{MIP-gap-gurobi}). Note that this definition is equivalent to $\UB = \frac{1}{1-p}\cdot\LB$, and we have $\LB\le \OPT$. Therefore, if the current gap is~$p\in[0,1)$, we obtain
\begin{align}
    \UB = \frac{1}{1-p}\cdot\LB \le \frac{1}{1-p}\cdot\OPT, \label{eq:gap}
\end{align}
so the currently best solution with objective value~$\UB$ must have gap~$p$ according to Definition~\ref{def:gap} (even though~$\OPT$ is usually still unknown).

\medskip

We now make use of the gap in the context of \CCG{} methods
and consider the case where the two-stage problem~\eqref{eq:mod} is not to be solved to optimality, but only up to some given non-zero target gap.
Our results are related to the termination of \rev{a \CCG{} method}.
First, suppose that, for some~$D \subseteq S$, the current master problem~$\MAS_D$ is solved to optimality with master solution~$(\tilde x, \tilde z)$.
Then, $\tilde z$ is an upper bound on the optimal objective value of the second-stage problem for each scenario~$s\in D$.
If there is no scenario in~$S$ for which the second-stage objective value is larger than~$\tilde z$, we can \revv{terminate} the algorithm because the current master solution~$(\tilde x, \tilde z)$ remains optimal
even when adding all scenarios to~$D$,
i.e., it is optimal for~$\MAS_S$.
Since~$\MAS_S$ is equivalent to~\eqref{eq:mod}, this means that~$\tilde x$ is an optimal first-stage solution of~\eqref{eq:mod}.
Proposition~\ref{prop:gappropagation} below shows that, even when solving the master problem~$\MAS_D$ only up to a non-zero gap, we may terminate the algorithm under these conditions and guarantee that the same gap is also reached for the original problem~\eqref{eq:mod}.

In the following, given a subset~$D \subseteq S$ of scenarios, we let~$\OPT_D$ denote the optimal objective value of~$\MAS_D$. Note that, since~$\MAS_S$ is equivalent to~\eqref{eq:mod}, this means that $\OPT_S\eqqcolon \OPT_{\RO}$ equals the optimal objective value of~\eqref{eq:mod}.

\begin{proposition}\label{prop:gappropagation}
  Suppose that~$\MAS_D$ has been solved up to a gap~$P \in [0, 1)$ for some subset~$D\subseteq S$, and let~$(\tilde x, \tilde z)$ denote the obtained master solution. If $Q(\tilde x, s) \le \tilde z$ for all~$s \in S$, then the first-stage solution~$\tilde x$ is a solution of~\eqref{eq:mod} with gap~$P$.
\end{proposition}
\begin{proof}
  Since~$(\tilde x, \tilde z)$ is a master solution of~$\MAS_D$ with gap~$P$, we have
  \begin{align}
    f(\tilde{x})+\tilde{z}\le \frac{1}{1-P}\cdot\OPT_D \le \frac{1}{1-P}\cdot\OPT_S = \frac{1}{1-P}\cdot\OPT_{\RO}. \label{eq:master-sol-gap}
  \end{align}
  Using the assumption that $Q(\tilde x, s) \le \tilde z$ for all~$s \in S$ together with~\eqref{eq:master-sol-gap} now yields the following upper bound on the objective value of the first-stage solution~$\tilde x$ in~\eqref{eq:mod}:
  \begin{align*}
    f(\tilde{x}) + \max_{s \in S} Q(\tilde x, s)
    \le f(\tilde{x}) + \tilde z
    \le \frac{1}{1-P}\cdot \OPT_{\RO}.
  \end{align*}
  This shows the claim.\qed
\end{proof}

Proposition~\ref{prop:gappropagation} shows that, if the current master solution~$(\tilde{x},\tilde{z})$ with gap~$p$ for~$\MAS_D$ satisfies $Q(\tilde x, s) \le \tilde z$ for all~$s \in S$, it also has gap~$p$ for~\eqref{eq:mod}. Therefore, if the goal is to solve~\eqref{eq:mod} up to a target gap of~$p$, we can terminate the solution process in this case.

Proposition~\ref{prop:zbound} below generalizes this statement to the case where only a target gap~$P$ larger than the gap~$p$ obtained for~$\MAS_D$ is to be reached for~\eqref{eq:mod}.
\rev{This is useful when the master problem has been solved with a gap strictly smaller than the user-defined target gap. This is indeed often the case with modern MIP solvers, which usually return solutions with gap strictly smaller than the user-defined target gap. Moreover, one can explicitly enforce a gap strictly smaller than the user-defined target gap by setting $\mu < 1$ in Algorithm~\ref{alg:general}, which sets the gap to be reached for~$\MAS_D$ to~$\mu P$.}

\begin{proposition}\label{prop:zbound}
  Let $P \in [0, 1)$ be the desired user-defined target gap for~\eqref{eq:mod} and $\mu \in [0,1]$ the master-gap factor.
  Suppose that~$\MAS_D$ has been solved up to a gap~$p$ with~$0\le p \le \mu P$ for some subset~$D\subseteq S$, and let~$(\tilde x, \tilde z)$ denote the obtained master solution.
  If $Q(\tilde x, s) \le \frac{1-p}{1-P}\cdot \tilde z + \frac{P-p}{1-P}\cdot f(\tilde x)$ for all~$s \in S$,
  then the first-stage solution~$\tilde x$ is a solution of~\eqref{eq:mod} with gap~$P$.
\end{proposition}
\begin{proof}
Combining~\eqref{eq:master-sol-gap} with the assumption that $Q(\tilde x, s) \le \frac{1-p}{1-P}\cdot \tilde z + \frac{P-p}{1-P}\cdot f(\tilde x)$ for all~$s \in S$ and using that $1+\frac{P-p}{1-P}=\frac{1-p}{1-P}$, we obtain
  \begin{align*}
    f(\tilde{x}) + \max_{s\in S}Q(\tilde{x},s)
    &\le f(\tilde{x}) + \frac{1-p}{1-P}\cdot\tilde{z} + \frac{P-p}{1-P}\cdot f(\tilde{x}) \\
    &= \left(1+\frac{P-p}{1-P}\right)\cdot f(\tilde{x}) + \frac{1-p}{1-P}\cdot\tilde{z} \\
    &= \frac{1-p}{1-P}\cdot\left(f(\tilde{x})+\tilde{z}\right) \\
    &\le \frac{1-p}{1-P}\cdot\frac{1}{1-p}\cdot \OPT_{\RO} \\
    &= \frac{1}{1-P}\cdot\OPT_{\RO},
  \end{align*}
    which shows the claim.\qed
\end{proof}

In the following, we denote the adjusted bound from Proposition~\ref{prop:zbound} by $\tilde z'\coloneqq \frac{1-p}{1-P}\cdot \tilde z + \frac{P-p}{1-P}\cdot f(\tilde x)$. Proposition~\ref{prop:zbound} generalizes Proposition~\ref{prop:gappropagation} by showing that a gap of at most~$P\geq p$ is reached for~\eqref{eq:mod} if the master solution~$(\tilde x, \tilde z)$ satisfies~$Q(\tilde x, s) \le \tilde z'$ for all~$s \in S$. Therefore, if~$P$ is the target gap the user aims to achieve for~\eqref{eq:mod}, we can terminate the solution process if $Q(\tilde{x}, s) \le \tilde z'$ for all~$s \in S$. Since~$\tilde z'\geq \tilde z$ and $Q(\tilde{x},s)\leq \tilde{z}$ already holds for all scenarios~$s\in D$ when~$(\tilde x, \tilde z)$ is the current master solution of~$\MAS_D$, this, in particular, means that we do not have to consider the second-stage problems for the scenarios~$s\in D$ at all during the search for a bad scenario. Moreover, comparing the current upper bounds of the second-stage problems for the scenarios in~$S\setminus D$ to~$\tilde z'$ instead of~$\tilde z$ has two advantages: First, it may allow an earlier termination of the overall solution process. Second, also an earlier termination of the solution process for each individual second-stage problem can be possible because the process can be stopped as soon as the upper bound is less than or equal to~$\tilde z'$.

\subsection{Algorithm description} \label{sec:alg-ours}

In this section, we present our new approximate \CCG{} method, called the \emph{approximate scenario bracketing procedure} (ASBP). The algorithm is embedded into the general \CCG{}-method structure from Algorithm~\ref{alg:general} and it combines and extends the ideas of the \CCG{} methods from \cite{toenissen11} and \cite{RODRIGUES2021499}. Moreover, it uses new ideas to further reduce computation time -- particularly in the case where~\eqref{eq:mod} is only to be solved up to a non-zero target gap. Our new implementation of \FBS() used in ASBP is shown in Algorithm~\ref{alg:ours}.

\begin{algorithm}
    \begin{small}
 	\caption{Implementation of $\FBS()$ in our method~(ASBP)}\label{alg:ours}
	\begin{algorithmic}[1]
 \Statex \textbf{Parameters:} Target gap~$P \in [0, 1)$, time MT used for solving~$\MAS_D$ in current iteration, parameters~$\text{TL}_\text{linear}$ and~$\text{TL}_\text{min}$
		\Statex \textbf{Input:} Scenario set~$D\subseteq S$,  master solution~$(\tilde{x},\tilde{z})$ of~$\MAS_D$ with gap~\mbox{$p \le P$}
		\Statex \textbf{Output:} A scenario~$s \in S$
		\Statex \textbf{Initialization:} $R \gets S \setminus D$, $\UB^s \gets +\infty,\; \LB^s \gets 0$,\; $\rho^s \gets \max(\text{TL}_\text{linear} \cdot \text{MT}, \text{TL}_\text{min}) \;\forall s \in S$.
		\For{$s \in S \setminus D$}\label{algline:o3}
		\State \parbox[t]{\dimexpr\linewidth-\algorithmicindent}{
			Run a heuristic for~$Q(\tilde x, s)$. If the heuristic finds a feasible solution, set $\UB^s$ to the corresponding second-stage cost. If the heuristic also returns a lower bound, set $\LB^s$ to  this lower bound.\label{algline:oheuristicstep}
		}
		\EndFor\label{algline:o8}
            \State $\tilde z' \gets \frac{1-p}{1-P} \cdot \tilde z + \frac{P-p}{1-P} \cdot f(\tilde x)$.\label{algline:odefz}
    	\While{True}\label{algline:owhilestart}
                \State $R \gets \{s \in R: \UB^s > \tilde z' \text{ and } \UB^s \ge \max_{r \in R} \LB^r\}$. \label{algline:o9}
                \If{$R = \emptyset$}
    			\State \Return any~$s \in D$.\label{algline:o15}
    		\EndIf
    		\State Let $k \in \argmax_{s\in R} \UB^s$. \label{algline:choosek}
    		\If{$|R| = 1$ and $\LB^k > \tilde z'$} \label{algline:oursz2}
    			\State \Return $k$.\label{algline:returnpop}
    		\EndIf
                \If{$\rho^k \le 0$ or $\LB^k = \UB^k$} \label{algline:ifreturntl}
    			\State \Return $k$.\label{algline:returntl}
    		\EndIf
    		\If{MIP model for $Q(\tilde x, k)$ not created yet}
    			\State Create MIP model for $Q(\tilde x, k)$.
    		\EndIf
    		\State Set time limit of MIP model for $Q(\tilde x, k)$ to~$\rho^k$.
    		\State \parbox[t]{\dimexpr\linewidth-\algorithmicindent}{
                    Start or continue solving the MIP model for~$Q(\tilde{x}, k)$ until the upper bound is strictly decreased, the lower bound is strictly increased, the time limit is reached, or the MIP solver decides that~$Q(\tilde{x}, k)$ is infeasible or solved to optimality.
                }\label{algline:osolvestep}\vspace{0.4em}
                \If{$Q(\tilde{x}, k)$ is infeasible} \State return $k$\EndIf
    		\State \parbox[t]{\dimexpr\linewidth-\algorithmicindent}{
                    Update~$\LB^k$ and~$\UB^k$ and decrease~$\rho^k$ by the time used by the MIP solver in this step.\label{algline:oupdateUBLB}
                }\vspace{0.4em}
        \EndWhile \label{algline:owhilend}
	\end{algorithmic}
    \end{small}
\end{algorithm}

Like the implementation of $\FBS()$ used in the ISAM (Algorithm~\ref{alg:toenissen}), our implementation shown in Algorithm~\ref{alg:ours} first applies a heuristic for the second-stage problem~$Q(\tilde{x}, s)$ of each scenario~$s$ (in our case, only for the scenarios~$s\in S\setminus D$ that could possibly be added to~$D$) to obtain upper bounds~$\UB^s$. Afterwards, we solve the second-stage problems in non-increasing order of~$\UB^s$, again as in Algorithm~\ref{alg:toenissen}. However, we use several new speed-up techniques that can often reduce the required computation time significantly.

\medskip

First, we also use a lower bound~$\LB^s$ on the optimal objective values of~$Q(\tilde{x},s)$ for each~$s\in S\setminus D$. Here, the first non-zero value for~$\LB^s$ can be obtained either when starting to solve the MIP model for~$Q(\tilde{x},s)$ for the first time, or already when running the heuristic for~$Q(\tilde{x},s)$ upfront in cases where the heuristic provides a lower bound. This is the case, for example, if the heuristic consists of a time-limited run of an MIP solver. The lower bounds are then used in our algorithm to speed up the computation of a bad scenario as follows: If $\UB^s \le \LB^{s'}$ for two scenarios~$s, s'$, we know that the optimal objective values of the corresponding second-stage problems satisfy $Q(\tilde x, s) \le Q(\tilde x, s')$, and can therefore disregard scenario~$s$ in the search for a worst scenario for~$\tilde{x}$ since scenario~$s'$ is provably as bad or worse.

\medskip

Second, whenever the solver finds a new best solution of the second-stage problem that is currently being solved, we decrease the upper bound $\UB^s$ of the corresponding scenario~$s$ accordingly and check whether~$s$ is still a worst scenario for the given first-stage solution~$\tilde{x}$ in~$S\setminus D$ with respect to~$\UB$. If this is still the case, we continue to solve the second-stage problem of scenario~$s$. Otherwise, we pause this problem's solution process and continue solving the second-stage problem for a scenario that is now worst for~$\tilde{x}$ in~$S\setminus D$ with respect to~$\UB$. This technique aims to avoid solving more than one scenario's second-stage problem to optimality, which can save significant computation time.

\medskip

Third, in order to prevent the algorithm from getting stuck on a particularly hard second-stage problem, we introduce a time limit~$\rho^s$ that upper bounds the total computation time invested for solving the second-stage problem of each scenario~$s\in S\setminus D$. If the time limit has already been reached in the previous iterations for the scenario that is currently worst for~$\tilde{x}$ in~$S\setminus D$ with respect to~$\UB$, we immediately return this scenario, which is then added to the scenario set~$D$ considered in the master problem for the following iteration of Algorithm~\ref{alg:general}. Note, however, that the task of finding a bad scenario for the current first-stage solution does not get harder with increasing cardinality of~$D$, i.e., with growing number of iterations
of Algorithm~\ref{alg:general}. In fact, this task actually becomes easier since we do not have to deal with scenarios in~$D$ anymore in this process. However, the size of the master problem and, therefore, the difficulty of solving it can grow significantly in each iteration. Thus, we increasingly focus on finding a worst scenario with growing number of iterations of Algorithm~\ref{alg:general}. We do this by gradually increasing the time limits for the second-stage problems after each iteration. Specifically, in order to relate the time limits to the increasing size of the master problem, we set the time limit to depend linearly on the time spent on solving the master problem in the current iteration of Algorithm~\ref{alg:general}. In Algorithm~\ref{alg:ours}, we allow to set this dependency factor $\text{TL}_\text{linear} \ge 0$ as well as a minimum time limit $\text{TL}_\text{min} \ge 0$ as parameters. Assuming $\text{TL}_\text{min} > 0$, the minimum time limit prevents the time limit from becoming zero when the master problem is solved in (almost) zero time. This can happen in some cases when using $\Init() = \emptyset$ to initialize~$D$ in Algorithm~\ref{alg:general}. Moreover, in our implementation, we allow the user to provide an already available first-stage solution for the initial set~$D$ to the algorithm. Also in this case, $\text{TL}_\text{min}$ is used as the time limit for each second-stage problem in the first run of $\FBS()$.

\medskip

Finally, an important new technique used in our method is that we exploit a user-defined non-zero target gap using the ideas presented in Section~\ref{sec:alg-gaps}. As in Algorithm~\ref{alg:rodriguez}, one can compare the current upper bounds on the second-stage costs of the different scenarios not only to each other, but also to the bound~$\tilde z$ from the provided master solution. In case of a non-zero user-defined target gap for~\eqref{eq:mod}, we have shown in Section~\ref{sec:alg-gaps} that using the alternative bound $\tilde{z}'\geq \tilde{z}$ from Proposition~\ref{prop:zbound} instead may allow for an earlier termination of the overall solution process, since all scenarios~$s$ with $\UB^s \le \tilde z'$ can be disregarded during the search for a bad scenario. In particular, because all scenarios~$s\in D$ have optimal value at most~$\tilde{z}\leq \tilde{z}'$ when~$(\tilde{x},\tilde{z})$ is the master solution provided in the input, we do not have to consider the second-stage problems for the scenarios~$s\in D$ at all during the search for a bad scenario.
From these observations, we derive the following termination condition for Algorithm~\ref{alg:ours}:
\begin{itemize}
\item A currently worst scenario~$s$ for~$\tilde{x}$ in~$S \setminus D$ with respect to~$\UB$ is the only scenario remaining to consider, and~$\UB^s$ cannot become smaller or equal to~$\tilde z'$ due to the lower bound~$\LB^s$. Then, scenario~$s$ is returned (Line~\ref{algline:returnpop}).
\item The time limit is reached for a
scenario~$s$ that is currently worst for~$\tilde{x}$ in~$S \setminus D$
with respect to~$\UB$, or such  a scenario's second-stage problem has already been solved to optimality. Then, scenario~$s$ is returned (Line~\ref{algline:returntl}).
\item All scenarios can be disregarded because $\UB^s \le \tilde{z}'$ for all~$s \in S\setminus D$ (and, thus, for all~$s \in S$). Then, any scenario~$s \in D$ is returned (Line~\ref{algline:o15}), which leads to immediate termination of Algorithm~\ref{alg:general} (as in Algorithm~\ref{alg:rodriguez}).
\end{itemize}

Note that, due to the use of the time limits, the ASBP is, in general, neither guaranteed to return a worst scenario nor a scenario whose optimal second-stage cost is larger than~$\tilde z'$.

\subsection{Analysis of our algorithm}\label{sec:alg-ours-analysis}

In this section, we prove the correctness of our approximate scenario bracketing procedure (ASBP) introduced in Section~\ref{sec:alg-ours}.
In Lemma~\ref{lemma:ouralg}, we focus on Algorithm~\ref{alg:ours}.
Theorem~\ref{th:general} then shows that Algorithm~\ref{alg:general} returns the desired result when using Algorithm~\ref{alg:ours} as the implementation of the subroutine \FBS(). For the proofs, we use the following assumptions on the heuristic applied in Line~\ref{algline:oheuristicstep} of Algorithm~\ref{alg:ours} and on the utilized MIP solver:

\begin{assumption}\label{assum:total}
\begin{enumerate}[{(a)}]
    \item The heuristic used in Line~\ref{algline:oheuristicstep} of Algorithm~\ref{alg:ours} terminates in finite time for any second-stage problem~$Q(x,s)$ with~$x\in\mathcal{X}$ and~$s\in S$.\label{assum:a}
    \item The MIP solver used in Algorithm~\ref{alg:ours} terminates in finite time for any second-stage problem~$Q(x,s)$ with~$x\in\mathcal{X}$ and~$s\in S$ and either returns an optimal solution or decides correctly that the problem is infeasible. In both cases, the upper and lower bound will be equal. \label{assum:b}
    \item The MIP solver used in Algorithm~\ref{alg:general} terminates in finite time for any master problem $\MAS_D$ with $D \subseteq S$ and \rev{returns a solution with gap $p\le \mu P$ or decides correctly that the problem is infeasible.} \label{assum:c}
\end{enumerate}
\end{assumption}

\begin{lemma}\label{lemma:ouralg}
    \begin{enumerate}[{(1)}]
        \item \revv{Under Assumptions~\ref{assum:total}~\eqref{assum:a} and~\ref{assum:total}~\eqref{assum:b}}, Algorithm~\ref{alg:ours} always terminates in finite time returning some scenario~$s\in S$. \label{lemma:ouralg1}
        \item If the scenario returned by Algorithm~\ref{alg:ours} is in~$D$, then $Q(\tilde x, s) \le \frac{1-p}{1-P}\cdot\tilde{z} + \frac{P-p}{1-P}\cdot f(\tilde{x})$ holds for all~$s \in S$.
          \label{lemma:ouralg2}
    \end{enumerate}
\end{lemma}
\begin{proof}
    \begin{enumerate}[{(1)}]
    \item
        The time spent in the for-loop in Lines~\ref{algline:o3} to~\ref{algline:o8} is finite due to the requirements on the heuristic in Assumption~\ref{assum:total}~\eqref{assum:a}. Thus, it only remains to show the finiteness of the while-loop in Lines~\ref{algline:owhilestart}--\ref{algline:owhilend}.
        Since, for each second-stage problem~$Q(\tilde{x},k)$, the MIP solver only needs finite time to either find an optimal solution or detect infeasibility by Assumption~\ref{assum:total}~\eqref{assum:b}, it can only spend finite time on each second-stage problem~$Q(\tilde{x},k)$ in Line~\ref{algline:osolvestep} overall (even without a time limit). Moreover, after an optimal solution of a second-stage problem~$Q(\tilde{x},k)$ has been found or infeasibility has been detected, we have $\LB^k=\UB^k$ by Assumption~\ref{assum:total}~\eqref{assum:b}. Then, if scenario~$k$ is ever selected again in Line~\ref{algline:choosek} in some future iteration and the algorithm does not terminate earlier, it will terminate and return scenario~$k$ in Line~\ref{algline:returntl}, due to the condition in Line~\ref{algline:ifreturntl}.
        Therefore, the while-loop terminates in finite time as claimed, and some scenario~$s\in S$ must then be returned since this happens for each possible termination condition of the while-loop.

        \medskip

    \item
        To simplify notation, let $\tilde{z}' = \frac{1-p}{1-P}\cdot \tilde{z} + \frac{P-p}{1-P}\cdot f(\tilde x)$ as in the algorithm.
        Since $(\tilde{x},\tilde{z})$ is the master solution provided in the input of the algorithm, the scenarios~$s \in D$ have optimal value at most $\tilde{z} \leq \tilde{z}'$, so the claim holds for these scenarios.

        It remains to show the claim for the scenarios in~$S \setminus D$. Since a scenario in~$D$ can only be returned in Line~\ref{algline:o15}, we know that~$R=\emptyset$ must hold in this case. Therefore, as~$R\coloneqq S\setminus D$ is chosen during initialization, each scenario from~$S\setminus D$ must be removed from~$R$ in some iteration of the algorithm.
        For each scenario~$s\in S\setminus D$, let~$\UB^{s,*}$ denote the value of~$\UB^s$ at the time when~$s$ is removed from~$R$ in Line~\ref{algline:o9}, and let~$t\in \argmax_{s\in S\setminus D} \UB^{s,*}$ be a scenario with the largest upper bound~$\UB^{t,*}$ at removal from~$R$. Then, since $Q(\tilde{x},s)\le \UB^{s,*}\le \UB^{t,*}$ for all~$s\in S\setminus D$, it suffices to show that $\UB^{t,*}\leq \tilde{z}'$.

        For the sake of a contradiction, suppose that $\UB^{t,*}>\tilde{z}'$. Then,
        $\UB^{t,*}<\max_{r\in R}\LB^r$ must hold
        when scenario~$t$ is removed from~$R$ in Line~\ref{algline:o9}. This implies that any scenario~$s\in \argmax_{r\in R}\LB^r$ satisfies~$\UB^{s,*}\ge \LB^s>\UB^{t,*}$ at this point, which contradicts the choice of scenario~$t$. This finishes the proof.\qed
    \end{enumerate}
  \end{proof}

Note that, due to the use of time limits for the second-stage problems, the inverse direction of statement~(2) in Lemma~\ref{lemma:ouralg} is not true in general. Note further that the proof of part~(1) does not use the existence of any time limits. Thus, the algorithm also terminates in finite time if no (finite) time limits are used for the second-stage problems.

\begin{theorem}\label{th:general}
  \revv{Under Assumption~\ref{assum:total}~\eqref{assum:c}, }when using Algorithm~\ref{alg:ours} to implement \linebreak\FBS(), Algorithm~\ref{alg:general} always returns a feasible solution of~\eqref{eq:mod} with gap at most~$P$ \revv{or decides correctly that the problem is infeasible}, independent of the subset~$D\subseteq S$ returned by \Init() during initialization.
\end{theorem}

\begin{proof}
    \revv{Unless infeasibility is detected,} one scenario is added to~$D$ in each iteration, and the \revv{termination condition} in Line~\ref{algline:generalterminationcrit} of Algorithm~\ref{alg:general} is satisfied after at most $|S|$~iterations. Each iteration runs in finite time due to Assumption~\ref{assum:total}~\eqref{assum:c} and Lemma~\ref{lemma:ouralg}~\eqref{lemma:ouralg1}.

    Algorithm~\ref{alg:general} \revv{terminates} if \revv{it detects infeasibility or} Algorithm~\ref{alg:ours} returns a scenario contained in~$D$. In the latter case, according to Lemma~\ref{lemma:ouralg}~\eqref{lemma:ouralg2}, we have $Q(\tilde x, s) \le \frac{1-p}{1-P} \cdot \tilde z + \frac{P-p}{1-P} \cdot f(\tilde x)$ for all $s \in S$. By Proposition~\ref{prop:zbound}, this implies that the first-stage solution~$\tilde x$ obtained in the current iteration is a solution of~\eqref{eq:mod} with gap at most~$P$.
	\qed
\end{proof}

\section{Applications}\label{sec:applications}

This section introduces two applications that are used to test our approximate scenario bracketing procedure. The first application is a robust capacitated location routing problem (RCLRP), which features a particularly hard second stage. In fact, the second-stage problem is a location routing problem itself and includes altering sizes of warehouses and solving a capacitated vehicle routing problem to determine delivery tours between the opened warehouses and the customers. In Section~\ref{sec:RCLRP}, we present a two-stage robust model for this problem that extends the (non-robust) model from~\cite{toro}.

The second application presented in Section~\ref{sec:BACASP} is a robust integrated berth allocation and quay crane assignment and scheduling problem (BACASP) studied in~\cite{RODRIGUES2021499}. We include this problem in our tests for two main reasons. First,~\cite{RODRIGUES2021499} propose the \emph{scenario reduction procedure} (SRP) presented in Section~\ref{sec:alg-previous} specifically to solve this problem.
Second, the authors of~\cite{RODRIGUES2021499} also develop a problem-specific combinatorial heuristic for the BACASP's second stage that can be used within our algorithm. This is in contrast to the first application (RCLRP), for which we use a short run of the MIP solver as the heuristic to quickly produce a rough ordering of the scenarios. Moreover, the BACASP is structurally different from the RCLRP in that it has most of its computational challenges in the first stage, i.e., solving the master problem is more demanding than solving the second-stage problem to identify bad scenarios. Therefore, the BACASP allows to evaluate how our algorithm performs when the second stage is comparatively easy.

\subsection{Robust capacitated location routing problem}\label{sec:RCLRP}

We first consider the robust capacitated location routing problem (RCLRP). The problem combines the capacitated facility location problem with the capacitated vehicle routing problem, and even the deterministic version is known to be NP-hard \cite{contardo2013computational,montoya2015literature}.
For more background and applications of the capacitated location routing problem (CLRP), we refer to the extensive literature overview in~\cite{toro} and the literature review in~\cite{PRODHON20141}.

\subsubsection{\rev{Problem description}}
We extend the CLRP to be (recoverable) robust against uncertainty in the customer demands, which are modeled using a finite set of scenarios. In our two-stage robust problem formulation, the first-stage decisions relate to which warehouse locations from a given set of potential candidates are to be opened and how to size them. For each scenario, the second-stage decisions involve allocating customers to warehouses and forming delivery tours from each warehouse to the assigned customers. In addition\rev{, as in~\cite{alvarez},} warehouse sizes can be increased or additional warehouses can be opened in the second stage at an increased cost.
An unlimited number of homogeneous, capacitated vehicles are available and each customer is supplied by only one vehicle. Each vehicle tour ends at the same warehouse that it starts from and visits only customers along the way.
In addition to cost minimization, the objective of the problem also incorporates the minimization of vehicle emissions resulting from the delivery tours chosen in the second stage.
A detailed description of the full model is provided in %
\if\arxivVersion0%
the supplementary material (Section~B).
\else%
\ref{sec:appendix-2-LRP}.
\fi

\subsubsection{Instances}
\label{LCRP_instances}
For our tests, we create test instances based on the data provided in~\cite{toro}.
Using the provided fixed and variable costs~$e$ and~$d$ of the warehouses in the first stage, however, usually leads to only one warehouse being opened. This choice is then not influenced too much by the uncertainty in the demands, and the first-stage decision usually stays the same over several iterations. To see a more prominent effect of the robust reformulation, we choose to multiply the provided fixed and variable first-stage warehouse costs~$e$ and~$d$ by a factor of~$1/40$, which leads to more warehouses being opened and more heterogeneity in the first-stage solutions.
The second-stage warehouse costs~$e'$ and~$d'$ are $50\%$ higher than the first-stage costs, i.e., $e' = 1.5 \cdot e$ and $d' = 1.5 \cdot d$.
Finally, for the demands, we use the deterministic values provided in~\cite{toro} as nominal demands. These are then randomly perturbed as follows in order to generate the scenarios: Each customer has a chance of~$50\%$ for having a demand of~$20\%$ above the nominal demand, and a $3\%$ chance to have a demand of zero (which corresponds to the customer location being closed).

\medskip

A specific instance in our tests is determined by the set~$I$ of potential warehouses, the set~$J$ of customers, the set~$S$ of scenarios, and an instance number, which is used to seed the random generator used for the scenario generation.

\medskip

To initialize the subset~$D$ of scenarios in Algorithm~\ref{alg:general} for this application, we simply use $\Init()\coloneqq\emptyset$. This leads to a first-stage solution in which no warehouses are opened at all (i.e., all~$w^0$ and~$a^0$ are set to zero), and the first scenario to be added to~$D$ is then determined in the first call of \FBS(). \rev{We note that, as shown in %
\if\arxivVersion0%
the supplementary material (Section~D)%
\else%
\ref{sec:appendix-4-results}%
\fi%
    , other possible implementations of $\Init()$ such as $\Init()\coloneqq\{s\}$ with a randomly chosen scenario or a scenario with maximum total demand perform similar or slightly worse.}

As a heuristic for the second-stage problem, we use a \rev{$0.1$\,s} run of the MIP solver.

\subsection{Berth allocation and quay crane assignment and scheduling problem}\label{sec:BACASP}
\revv{Our second application is the robust integrated berth allocation and quay crane assignment and scheduling problem (BACASP). The deterministic version of the problem is considered in~\cite{AGRA2018138}, and the two-stage robust version is studied in~\cite{RODRIGUES2021499}.}

\subsubsection{\rev{Problem description}}
The BACASP is an operational planning problem for ports in the context of maritime freight transportation. Vessels of different lengths arrive at a port with uncertain arrival times, need to berth at a specific position (moor at their allotted place at the wharf) without intersecting other vessels, and \revv{must} be unloaded by quay cranes that can translate along the wharf. Each crane has a certain processing rate which is the amount of cargo volume that it can unload from the vessel in a certain time period. The vessel is \emph{unloaded}, when the sum of cargo volume processed by all assigned cranes equals the total cargo volume of the vessel at berthing time. When the vessel is unloaded, it immediately departs. The task is to allocate each vessel to a berth range and determine the scheduling of available cranes with the goal of minimizing the total completion time, i.e., the sum over all vessels of the difference between the departure time of the vessel and its arrival time. Here, the departure times are influenced by the selected berth allocation and quay crane assignment and schedule, e.g., since vessels might have to wait outside of the port when their assigned berth is still occupied by another vessel that is still being unloaded.
To efficiently manage the port, and to minimize the time required to process each vessel, mixed-integer programming formulations of the problem have been proposed \cite{AGRA2018138}. Moreover, several authors have studied similar problems under uncertainty in the arrival times by using two-stage robust formulations, such as the robust cyclic berth planning of container vessels \cite{hendriks2010robust} or the berth allocation problem \cite{liu2020two}. The BACASP with uncertain arrival times was first studied in the context of robust optimization in \cite{RODRIGUES2021499}. The proposed model determines the berth allocation in the first stage and the quay crane assignment and scheduling in the second stage of the robust model.

In our tests, we use the two-stage robust mixed-integer programming model of the BACASP presented in \cite{RODRIGUES2021499}. \rev{For completeness, we state the model in %
\if\arxivVersion0%
the supplementary material (Section~C).
\else%
\ref{sec:appendix-3-BACASP}.
\fi
For a more detailed description, we refer to the original paper \cite{RODRIGUES2021499}}.

\subsubsection{Instances} \label{BAQ_instances}
For our tests, we were supplied with the original instances used in \cite{RODRIGUES2021499}.
For each number of vessels, 10 randomly generated instances with random nominal arrival times, random cargo volumes, and random vessel lengths are available. The crane-related data (number of cranes, range in which they can operate, and processing rate) stay the same over all instances. We only use homogeneous instances, meaning that all cranes have the same processing rate.

To generate the scenarios, the authors of \cite{RODRIGUES2021499} used a special case of a multiple constrained budgeted uncertainty set. Assume that the arrival time of a vessel~$k$ in scenario~$s$ is given by $A^s_k \in \R_{\ge 0}$, and the vector of all arrival times in scenario~$s$ is $A^s$. The nominal arrival times for each vessel~$k$ are given by $A_k$, the maximum allowed delay of a vessel~$k$ is $\hat A_k$, $N$ is the total number of vessels and $V=\{1,\dots,N\}$ is the set of vessels. The set of vessels is split into three subsets of near-equal size and, in each of these subsets, at most one vessel can be delayed by either the maximum allowed delay or half of this maximum allowed delay. In order to define the uncertainty set formally, we first define the set~$\mathcal{D}$ of deviation vectors as
\begin{align*}
    \mathcal{D} \coloneqq \left\{\delta \in \{0,0.5,1\}^V :\;
	\sum_{k=1}^{\left[N/3 \right]-1} \lceil \delta_k \rceil \le 1,
	\sum_{k=\left[N/3 \right]}^{2\,\left[N/3 \right]-1} \lceil \delta_k \rceil \le 1,
	\sum_{k=2\,\left[N/3 \right]}^{N} \lceil \delta_k \rceil \le 1 \right\},
\end{align*}
where~$[\cdot]$ is the operator that rounds to the closest integer and $\lceil\cdot\rceil$ denotes rounding up to the next integer.
The uncertainty set~$\mathcal{A}$ for the arrival times is then given as follows:
\begin{align*}
	\mathcal{A} \coloneqq \left\{ A^s = \begin{pmatrix}
           A_1 + \hat A_1 \,\delta_1 \\
           \vdots \\
           A_N + \hat A_N \,\delta_N
         \end{pmatrix} : \delta \in \mathcal{D} \right\}.
\end{align*}
The set~$S=\{1,\dots,|\mathcal{A}|\}$ of scenarios is then of the same cardinality as~$\mathcal{A}$, and each element~$s$ in~$S$ corresponds to an element in~$\mathcal{A}$, named~$A^s$.

This combinatorial way of defining the scenarios results in a maximum possible number of scenarios that depends on the number of vessels. Table~\ref{tab:bacaspscenarios} shows this maximum possible number for each number of vessels.

\begin{table}[h!]
	\centering
	\begin{tabular}{c|cccccccccccccc}
		$N$ & 6 & 7 & 8 & 9 & 10 & 11 & 12 & 13 & 14 & 15 \\
		\hline
		$|S|$ & 125 & 175 & 245 & 343 & 441 & 567 & 729 & 891 & 1089 & 1331 \\
	\end{tabular}
 \caption{Maximum possible number of available scenarios for each number of vessels in the BACASP instances.}\label{tab:bacaspscenarios}
\end{table}

To initialize the subset~$D$ of scenarios in Algorithm~\ref{alg:general} for this application, we use \rev{$\Init()\coloneqq \{s\}$, where~$s$ results from the so-called} \emph{slack reduction heuristic}, which is identified as superior to initialization with a randomly chosen scenario in \cite{RODRIGUES2021499}. Basically, this returns the scenario for which the expected port congestion is highest, where the expected congestion is calculated by considering the minimal distances in arrival times between two consecutive vessels. \rev{As a heuristic for the second-stage problem, we use the problem-specific \emph{scenario evaluation heuristic} (SEH) as proposed in \cite{RODRIGUES2021499}.}

\section{Computational experiments}\label{sec:computational-exps}

In Section~\ref{sec:comparison}, we first compare the performance of the ASBP to the two recent column-and-constraint generation methods ISAM \cite{toenissen11} and SRP \cite{RODRIGUES2021499} on the robust capacitated location routing problem (RCLRP) described in Section~\ref{sec:RCLRP} and the robust integrated berth allocation and quay crane assignment and scheduling problem (BACASP) described in Section~\ref{sec:BACASP}. Afterwards, in Section~\ref{sec:variants}, we analyze the impact of the different techniques employed in ASBP by comparing its performance to several variants in which one of these techniques is disabled in each case.

\medskip

All experiments were conducted on a Linux system running Ubuntu~23.04. The hardware configuration comprises an AMD EPYC~7542 processor with 32~cores and 64~threads, operating at a base clock frequency of 2.9~GHz. The system is equipped with 512~GB of RAM. The experiments were implemented using Python~3.11 and Gurobi~11.0.3. For each instance, the problem was solved using a single thread to ensure consistent performance and resource availability across all experiments.

\subsection{Comparison to other \CCG{} methods }\label{sec:comparison}

All computational experiments were conducted using target gaps\rev{~$P$} of \{0\%, 5\%, 10\%\}.
\rev{In the ISAM and the SRP, the master-gap factor is set to~$\mu=1$ since choosing~$\mu<1$ would lead to the problem being solved with a smaller-than-desired gap of~$\mu P < P$ in these methods. For our ASBP, we tested master-gap factors~$\mu$ of \{0.0, 0.25, 0.5, 0.75, 1\} and identified $\mu=0.5$ as the best choice in both RCLRP and BACASP (see %
\if\arxivVersion0%
the supplementary material (Section~D)%
\else%
\ref{sec:appendix-4-results}%
\fi%
    ). Therefore, we use $\mu=0.5$ in the ASBP in our comparison.} A 30-minute time limit was used for small instances, while large instances were given a \num{3}-hour time limit. \rev{For the time limit parameters used in \FBS() in the ASBP, we use $\text{TL}_\text{linear}=\num{2}$ and $\text{TL}_\text{min}=\SI{1}{s}$.} For each combination of parameters, 10~instances were solved. For all RCLRP instances and the small BACASP instances, we ran two repetitions per instance. Due to the elevated memory requirements of the large BACASP instances, only few instances could be run simultaneously and, thus, only one repetition has been run. General details regarding the instance generation can be found in Section~\ref{LCRP_instances} for the RCLRP and Section~\ref{BAQ_instances} for the BACASP.

\medskip

The RCLRP experiments tested small and large instances, each with 5~warehouses. Small instances involved customer numbers of \{12, 16, 18, 20\}, with \{16, 64\} scenarios, leading to $160$~runs per target gap and algorithm. Large instances considered customer numbers of \{20, 22, 24, 26\}, with \{32, 64\} scenarios, resulting in $160$~runs per target gap and algorithm.

\medskip

The BACASP experiments tested vessel counts of \{6, 7, 8, 9\} for small instances, and \{10, 11, 12, 13\} for large instances. For \rev{small instances}, \{16, 64, $|S|$\} scenarios were tested, where~$|S|$ denotes the maximum possible number of scenarios for the given number of vessels (see Table~\ref{tab:bacaspscenarios}). \rev{For large instances, we tested 64 and $|S|$ scenarios.} This resulted in $240$~runs per target gap and algorithm for the small instances, and \rev{$80$}~runs per target gap and algorithm for the large instances.

\medskip

Figures~\ref{fig:lrp} and~\ref{fig:baq} show the total runtime (processor time) in seconds on a logarithmic scale against the percentage of solved instances for the three different methods ASBP (our method), ISAM \cite{toenissen11}, and SRP \cite{RODRIGUES2021499} across the different target gaps and instance sizes.

\medskip

The performance comparisons for the RCLRP are shown in Figure~\ref{fig:lrp}. This problem involves a hard second stage, and 46\% of the total computation time used by Gurobi (excluding the use as heuristic) is spent on average for solving the second-stage problem.
As Figure~\ref{fig:lrp} shows, the ASBP consistently outperforms the ISAM and the SRP by a wide margin in the RCLRP across all target gaps and instance sizes. For both instance sizes, the advantage of the ASBP grows considerably with increasing target gap, which can be seen from both a larger percentage of instances solved within the time limit as well as higher percentages solved at earlier times. For the large instances, almost no instances could be solved within the 3-hour time limit by the ISAM and the SRP, while our method ASBP solves 13.8\%, 80\%, and 83.1\% of the instances for target gaps of 0\%, 5\%, and 10\%, respectively.

\begin{figure}
\includegraphics{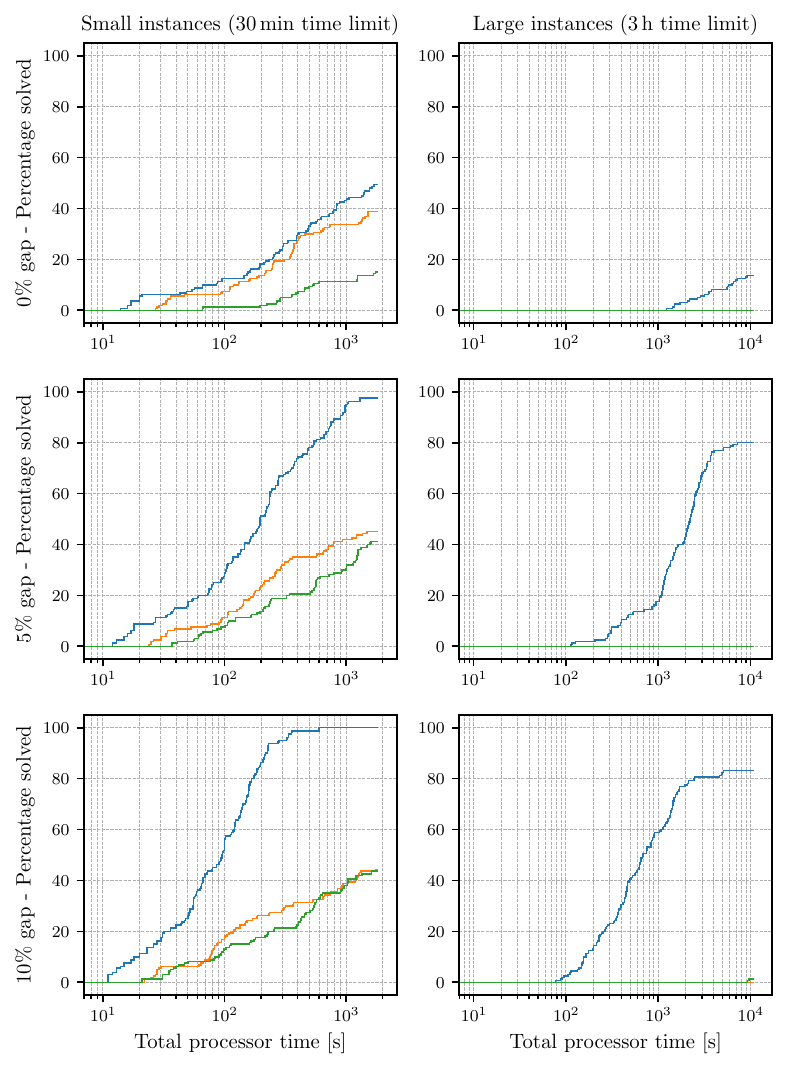}
\caption{Performance plots of our ASBP with $\mu=0.5$ in blue, the ISAM \cite{toenissen11} in orange, and the SRP \cite{RODRIGUES2021499} in green for the RCLRP.}
\label{fig:lrp}
\end{figure}

Figure~\ref{fig:baq} shows the performance comparisons for the BACASP. Here, the second-stage problem is much easier to solve, and only 31\% of the total computation time used by Gurobi is spent on average for solving the second-stage problem. \revv{As Figure~\ref{fig:baq} shows, in the BACASP, the ASBP performs similarly to the SRP in the runs with zero target gap. The SRP was designed specifically for the BACASP, which explains its good performance in this application. The runs with non-zero target gap, however, show clear advantages of ASBP, as it consistently outperforms both the SRP and ISAM by a notable margin. However, the performance difference is smaller than it was for the RCLRP. This can be explained by the fact that our new ASBP is mainly targeted at problems with a challenging second stage, which offer more potential for saving time during the identification of bad scenarios. Therefore, the rather easy-to-solve second-stage problem in the BACASP reduces the impact of the new techniques for faster identification of bad scenarios in the ASBP. Still, our new method shows a strong performance competitive with the problem-specific SRP in the BACASP, and it considerably outperforms the ISAM.}

\begin{figure}
\includegraphics{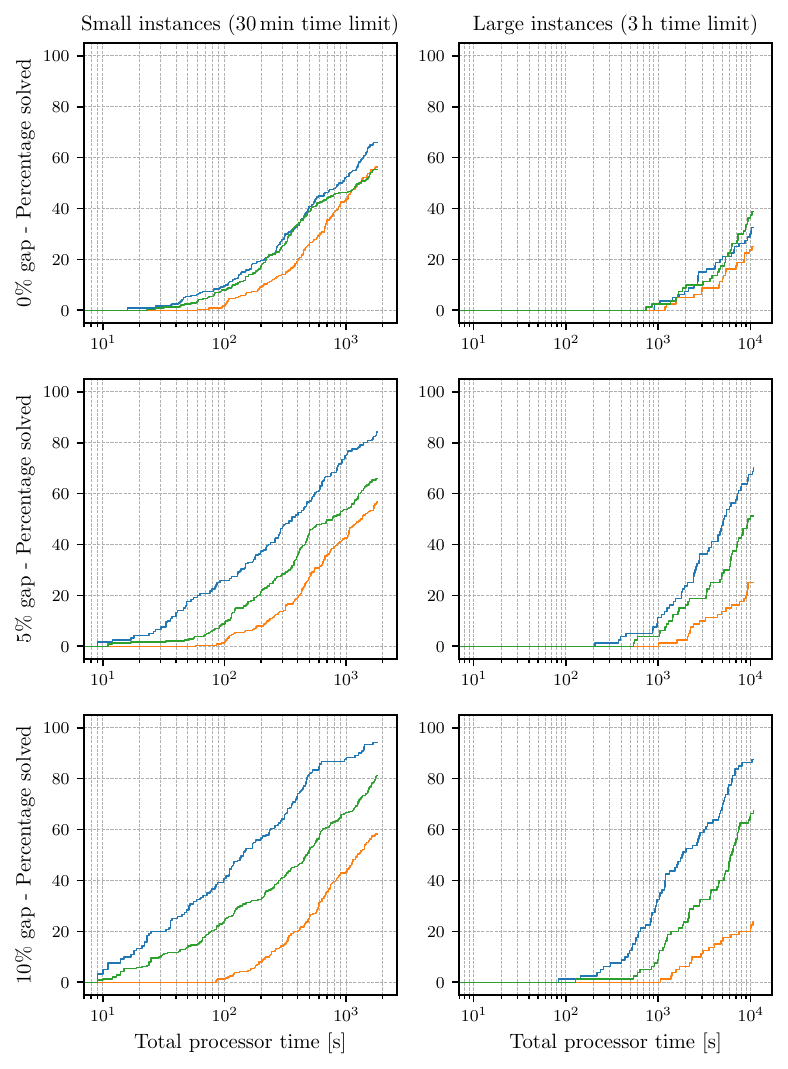}
\caption{Performance plots of our ASBP with $\mu=0.5$ in blue, the ISAM \cite{toenissen11} in orange, and the SRP \cite{RODRIGUES2021499} in green for the BACASP.}
 \label{fig:baq}
\end{figure}

\subsection{Comparison of ASBP variants}\label{sec:variants}

We \revv{now perform an ablation study and} analyze the impact of the different techniques employed in ASBP by comparing its performance to several variants in which one of these techniques is disabled in each case. Here, we exclusively used the RCLRP since its hard second stage offers more potential for the different techniques to have an impact on the overall performance of the method. To obtain an insightful comparison, we used a broad range of instance sizes with 5~warehouses, \{12, 18, 22, 24\} customers, and \{16, 64\} scenarios. As in the previous section, we generated 10~instances for each combination of parameters, and solved two repetitions of each instance. We used a time limit of 100~minutes for each instance, and target gaps of \{0\%, 5\%, 10\%\}.

\medskip

Table~\ref{tab:lrp-table} compares the average runtime and the number of runs solved within the 100-minute time limit for the following variants of the ASBP:
\begin{itemize}
    \item \textbf{ASBP:} Standard version of the ASBP (Algorithm~\ref{alg:general} with Algorithm~\ref{alg:ours}).
    \item \textbf{no LB:} The ASBP without using lower bounds for second-stage problems. Here, we let~$\LB^s = 0$ throughout the algorithm for all~$s\in S$ whose second-stage problem is not solved to optimality yet, and we only set~$\LB^s$ to the optimal objective value of~$Q(\tilde{x},s)$ in Line~\ref{algline:oheuristicstep} or Line~\ref{algline:oupdateUBLB} once~$Q(\tilde{x},s)$ has been optimally solved. Therefore, scenarios can only be excluded from further consideration in Line~\ref{algline:o9} due to their upper bound being smaller than the lower bound of an already optimally solved scenario (or smaller than or equal to~$\tilde z'$).
    \item \textbf{no ZB:} The ASBP without comparisons of the upper bounds of the scenarios to~$\tilde{z}'$. Here, we replaced the definition of~$\tilde{z}'$ in Line~\ref{algline:odefz} in our algorithm by $\tilde{z}' \gets 0$ and considered all scenarios in~$S$ during the search for a bad scenario (as the scenarios~$s\in D$ are not automatically excluded from the start due to their optimal second-stage cost being at most~$\tilde{z}$).
    \item \textbf{no TL:} The ASBP without time limits for the second-stage problems. Here, we replaced the initialization of the time limits~$\rho^s$ by $\rho^s \gets +\infty$. In particular, this means that Algorithm~\ref{alg:ours} only terminates if an actually worst scenario has been found or is already contained in~$D$.
\end{itemize}

\begin{center}
\begin{table}[h]
    \centering

    \begin{adjustbox}{width=\textwidth}
\newcolumntype{R}[1]{>{\raggedleft\arraybackslash}m{#1}}
\begin{tabular}{rrr|*{4}{R{18mm}}|*{4}{R{15mm}}}
\toprule
     &    &    &  \multicolumn{4}{c}{Average total runtime} & \multicolumn{4}{|c}{Number of solved runs within time limit} \\
Gap & Cust. & Sc. & ASBP &                   no LB &                   no ZB &                   no TL & ASBP &               no LB &              no ZB &          no TL \\
\midrule
0.00 & 12 & 16 &   134 &    137 &    171 &    150 &    20 &     20 &     20 &     20 \\
     &    & 64 &   866 &    871 &    903 &    760 &    18 &     18 &     18 &     18 \\
     & 18 & 16 &  4385 &   4048 &   4257 &   3258 &     8 &      9 &      8 &     10 \\
     &    & 64 &  6000 &   6000 &   6000 &   5511 &     0 &      0 &      0 &      2 \\
     & 22 & 16 &  5232 &   5228 &   5308 &   5160 &     4 &      4 &      4 &      4 \\
     &    & 64 &  6000 &   6000 &   6000 &   6000 &     0 &      0 &      0 &      0 \\
     & 24 & 16 &  4808 &   5133 &   5153 &   4975 &     7 &      5 &      6 &      8 \\
     &    & 64 &  5703 &   5761 &   5798 &   5744 &     2 &      2 &      3 &      3 \\
     \hline &&&&&&&&&& \\[-2ex]
     \multicolumn{3}{c|}{Average 0\% gap}     &  4141 &   4147 &   4199 &   3945 &   7.4 &    7.2 &    7.4 &    8.1 \\
     \midrule
0.05 & 12 & 16 &    53 &     59 &     82 &     65 &    20 &     20 &     20 &     20 \\
     &    & 64 &   187 &    191 &    276 &    139 &    20 &     20 &     20 &     20 \\
     & 18 & 16 &   289 &    291 &   1580 &    487 &    20 &     20 &     18 &     20 \\
     &    & 64 &   419 &    378 &   2656 &    811 &    20 &     20 &     20 &     20 \\
     & 22 & 16 &  1650 &   1668 &   3491 &   3255 &    16 &     16 &     14 &     16 \\
     &    & 64 &  2134 &   2360 &   5292 &   4032 &    18 &     18 &     10 &     15 \\
     & 24 & 16 &  1627 &   1580 &   4148 &   3083 &    18 &     18 &     12 &     16 \\
     &    & 64 &  2707 &   2479 &   4986 &   5217 &    17 &     18 &      8 &      6 \\
     \hline &&&&&&&&&& \\[-2ex]
     \multicolumn{3}{c|}{Average 5\% gap}    &  1133 &   1126 &   2814 &   2136 &  18.6 &   18.8 &   15.2 &   16.6 \\
     \midrule
0.10 & 12 & 16 &    27 &     29 &     54 &     40 &    20 &     20 &     20 &     20 \\
     &    & 64 &    81 &     84 &    200 &     99 &    20 &     20 &     20 &     20 \\
     & 18 & 16 &   124 &    103 &    975 &    413 &    20 &     20 &     20 &     20 \\
     &    & 64 &   199 &    251 &   2648 &    766 &    20 &     20 &     19 &     20 \\
     & 22 & 16 &   825 &   1453 &   2988 &   3004 &    18 &     16 &     16 &     18 \\
     &    & 64 &  1318 &   1326 &   4230 &   3582 &    18 &     18 &     14 &     18 \\
     & 24 & 16 &   808 &    847 &   4129 &   2910 &    18 &     18 &     13 &     16 \\
     &    & 64 &  1134 &   1141 &   4673 &   4990 &    18 &     18 &     12 &      7 \\
     \hline &&&&&&&&&& \\[-2ex]
     \multicolumn{3}{c|}{Average 10\% gap}     &   565 &    654 &   2487 &   1976 &  19.0 &   18.8 &   16.8 &   17.4 \\
     \midrule
     \multicolumn{3}{c|}{Overall average} &  1946 &   1976 &   3167 &   2685 &  15.0 &   14.9 &   13.1 &   14.0 \\
\bottomrule
\end{tabular}
\end{adjustbox}

    \caption{Comparison of ASBP \rev{(with $\mu=0.5$)} variants for the RCLRP.
    Each provided average total runtime is the average over 10~instances and 2~repetitions per instance (i.e., 20~data points are used for each reported average). The time limit is set to 100~minutes. \rev{All time measurements are given in seconds.}}
    \label{tab:lrp-table}
\end{table}
\end{center}

The standard version of the ASBP achieved the lowest average runtime of 1946 seconds and solved 15.0 out of 20 runs within the time limit on average. The variants \emph{no ZB} and \emph{no TL} performed much worse and required 3167 seconds and 2685 seconds on average and only solved an average of 13.1 and 14.0 runs, respectively. In contrast, the variant \emph{no LB} performed only slightly worse than the standard version of the ASBP with an average runtime of 1976 seconds and an average of 14.9 runs solved within the time limit.

\medskip

Comparing the results for individual choices of an instance size and a target gap shows that, while the standard ASBP performed best in most cases, there are some cases where other variants showed a slightly better performance. This better performance of other variants in isolated cases can mostly be attributed to the non-deterministic nature of the employed MIP solver Gurobi. In particular, the quality of the solutions and lower bounds provided by the \rev{0.1-second} run of Gurobi used as the heuristic in Line~\ref{algline:oheuristicstep} can vary substantially between runs, which can affect the overall runtime of the algorithm significantly since it may change the order in which the second-stage problems are considered afterwards.

\medskip

Overall, it can be observed that the use of time limits and the comparisons of upper bounds of the scenarios to~$\tilde{z}'$ had a larger impact on the runtime than excluding scenarios using their lower bounds. This can partly be explained by the observation that, in many cases, the time limit for a currently worst scenario was exhausted before its lower bound could exclude all other scenarios. Moreover, as is to be expected, comparisons of upper bounds of the scenarios to~$\tilde{z}'$, which, in particular, exploit a non-zero target gap, had larger influence for larger values of the target gap.

\section{Conclusion}

Two-stage and recoverable robust optimization problems provide flexible modeling opportunities, but at the same time, are computationally very challenging to solve. A widely applied principle in this context is that of column-and-constraint generation, where one starts with a simplification of the problem and iteratively makes it more complex until an optimal solution for the original problem has been achieved. In this paper, we introduced, analyzed, and tested several techniques to improve this process. Our resulting method is particularly designed for the case that the second-stage problem is difficult to solve, which means that already identifying a worst-case scenario becomes a computational burden.
We introduced our method in the context of (linear) mixed-integer programs, but it can be directly extended to non-linear problems by adjusting our assumptions. By comparing to other state-of-the-art column-and-constraint generation methods from the literature, we were able to show that our approach is particularly strong if the second-stage problem is indeed hard to solve, but we still remain competitive to the best known methods even if this is not the case.

\revv{An assumption we made is that feasible solutions for the second-stage problem can be found easily by means of heuristics or an MIP solver.}
In possible further research, it would be interesting to consider column-and-constraint generation methods that still work if finding a feasible solution in the second stage is already hard. An interesting further challenge is to include machine learning methods within our scenario addition framework to achieve further speed-ups if data is available on the performance of our method on previous runs (an assumption that we did not need to make in this paper).


\section*{Statements and declarations}

\subsection*{Funding}
\noindent
This work was supported by the German Federal Ministry of Education and Research (BMBF) [grant number~05M22WTA].

\subsection*{Author contributions}
{\parindent0pt
\textbf{Marc Goerigk:} Methodology, Validation, Formal analysis, Writing - Review \& Editing

\textbf{Johannes Kager:} Methodology, Software, Validation, Formal analysis, Investigation, Data Curation, Writing - Original Draft, Writing - Review \& Editing, Visualization.

\textbf{Dorothee Henke:} Methodology, Validation, Formal analysis, Writing - Review \& Editing

\textbf{Fabian Schäfer:} Methodology, Validation, Formal analysis, Writing - Review \& Editing

\textbf{Clemens Thielen:} Conceptualization, Methodology, Validation, Formal analysis, Writing - Review \& Editing, Supervision, Project Administration, Funding acquisition.
}

\subsection*{Declaration of interest}
\noindent
Declarations of interest: none

\subsection*{Acknowledgments}
\noindent
We gratefully acknowledge Filipe Rodrigues for generously providing us with the data and code from their work \cite{RODRIGUES2021499}, and Eliana M.\ Toro for generously providing us with the data from their work \cite{toro}.

\subsection*{Data availability}
\noindent
The implementations of the algorithms and models as well as the data sets used in the computational experiments are published at \url{https://doi.org/10.5281/zenodo.17363541}.


\newcommand{\etalchar}[1]{$^{#1}$}

\if\arxivVersion1
 \newpage
 \appendix
 \section{ISAM and SRP with non-zero target gap}\label{sec:appendix-1-proofs}

We provide a short proof of correctness for the ISAM from~\cite{toenissen11} and the SRP from~\cite{RODRIGUES2021499} if these are implemented with a non-zero target gap~$P$ in our general \CCG{}-method structure (Algorithm~1), which will allow us to compare our algorithm to these methods also for non-zero target gaps. \rev{Note that, without loss of generality, we assume a master-gap factor of~$\mu=1$ in the ISAM and the SRP since using $\mu<1$ has exactly the same effect in these methods as decreasing the target gap~$P$ itself to~$\mu P$.}

\setcounter{theorem}{6}
\begin{proposition}\label{prop:nonzeroinprevious}
  Suppose that Algorithm~1 is applied with any target gap~$P\in[0,1)$  \rev{and master-gap factor~$\mu = 1$,} and \FBS() is implemented by either Algorithm~2 or Algorithm~3. Then the first-stage solution~$\tilde{x}$ returned by Algorithm~1 is a solution of~\textup{(2-RO)} with gap~$P$.
\end{proposition}
\begin{proof}
Algorithm~1 terminates if and only if the scenario returned by \linebreak\FBS() is contained in~$D$. We now show that, when Algorithm~2 or Algorithm~3 returns a scenario~$k \in D$, we have $Q(\tilde x, s) \le \tilde z$ for all~$s \in S$. The claim then follows directly from Proposition~2.

First consider Algorithm~2.
Since Algorithm~2 returns a scenario~$k \in \argmax_{s\in S} \UB^s$ whose second-stage problem has been solved to optimality, we have $\UB^k = Q(\tilde x, k)$.
This means that $Q(\tilde{x}, s) \le \UB^s \le \UB^k = Q(\tilde{x}, k)$ holds for all~$s \in S$.
If~$k\in D$, then $Q(\tilde{x}, k) \le \tilde{z}$ holds, so we obtain $Q(\tilde x, s) \le \tilde z$ for all~$s \in S$ as required.

Now consider Algorithm~3.
The algorithm only returns a scenario in~$D$ if no scenario~$s \in S\setminus D$ with $Q(\tilde{x}, s) > \tilde{z}$ is found, i.e., if $Q(\tilde{x}, s) \le \tilde{z}$ for all~$s \in S\setminus D$. Since $Q(\tilde{x}, s) \le \tilde{z}$ holds for all~$s\in D$, this implies that $Q(\tilde x, s) \le \tilde{z}$ for all~$s \in S$ in this case.
\end{proof}

 \section{RCLRP Model Formulation}\label{sec:appendix-2-LRP}


To model the RCLRP formally, we extend the (non-robust) CLRP formulation from~\cite{toro} to a two-stage robust model. The underlying structure is a complete, directed graph whose node set~$V=I\cup J$ is partitioned into a set~$I$ of potential warehouses~(WH) and a set~$J$ of customers. Each directed arc~$(v_1,v_2)$ is equipped with a non-negative traversal cost~$c_{v_1,v_2}$ and two non-negative values~$\alpha_{v_1,v_2}$ and~$\gamma_{v_1,v_2}$. The value $\alpha_{v_1,v_2}$ represents the cost of emissions produced by an empty vehicle traversing the arc~$(v_1,v_2)$, while $\gamma_{v_1,v_2}$ is a conversion factor that describes the arc-dependent additional cost of emissions per ton of cargo that a traversing vehicle is loaded with.
Furthermore, each potential warehouse~$i\in I$ has non-negative fixed costs~$e_i$ and~$e'_i>e_i$ for opening it in the first and in the second stage, respectively, as well as non-negative variable costs~$d_i$ and~$d'_i>d_i$ describing the cost per unit of warehouse size established in the first and in the second stage. The maximum possible capacity of warehouse~$i$ is given by~$A_i\geq 0$, and~$L\geq 0$ and~$F\geq 0$ denote the vehicle capacity and the fixed cost per used vehicle (i.e., per vehicle tour), respectively.

Variables with superscript~$0$ denote first-stage decisions, while variables with superscript~$s$ denote second-stage decisions for a scenario~$s \in S$. Each scenario is determined by the corresponding demand vector~$\beta^s \in \R^{J}_{\ge 0}$ that specifies a non-negative demand~$\beta^s_j\leq L$ for each customer~$j\in J$. The problem parameters and decision variables are given as follows:

\medskip

\noindent
\textbf{Parameters:}\\
\noindent
\def\arraystretch{1.4}
\begin{tabular}{wc{1cm}|l}
	$I$ & set of potential WHs \\
	$J$ & set of customers \\
	$V$ & set of nodes ($V=I\cup J$) \\
	$S$ & finite set of scenarios \\
	$\beta^s_j$ & demand of customer~$j$ in scenario~$s$ \\
	$e_i, e'_i$ & fixed cost for opening WH~$i$ in first or second stage \\
        $d_i, d'_i$ & cost per unit of size of WH~$i$ in first or second stage \\
        $c_{v_1,v_2}$ & traversal cost of arc~$(v_1,v_2)$ for $v_1, v_2 \in V$ \\
        $\alpha_{v_1,v_2}$ & cost of emissions of empty vehicle on arc~$(v_1, v_2)$ for $v_1, v_2 \in V$ \\
        $\gamma_{v_1,v_2}$ & additional cost of emissions per ton cargo on arc~$(v_1, v_2)$ for $v_1, v_2 \in V$ \\
        $A_i$ & maximum possible capacity of WH~$i$\\
	$L$ & vehicle capacity \\
	$F$ & fixed cost per used vehicle (i.e., per vehicle tour)
\end{tabular}

\bigskip

\noindent
\textbf{First-stage decision variables:}

\noindent
\begin{tabular}{wc{0.18\linewidth}|l}
	$w^0_i \in \{0,1\}$ & $1$ if WH~$i$ is opened in the first stage, $0$ otherwise\\
	$a_i^0 \in \R_{\ge 0}$ & chosen size of WH~$i$ in the first stage\\
\end{tabular}

\bigskip

\noindent
\textbf{Second-stage decision variables for scenario~$s \in S$:}

\noindent
\begin{tabular}{wc{0.18\linewidth}|l}
	$w^s_i \in \{0,1\}$ & $1$ if WH~$i$ is open in the second stage, $0$ otherwise\\
	$a_i^s \in \R_{\ge 0}$ & chosen size of WH~$i$ in the second stage\\
	$r_{v_1,v_2}^s \in \{0,1\}$ & $1$ if arc $(v_1, v_2)$ for~$v_1, v_2 \in V$ is used, $0$ otherwise\\
	$t_{v_1,v_2}^s \in \R_{\ge 0}$ & tons of cargo transported on arc~$(v_1, v_2)$ for~$v_1, v_2 \in V$\\
	$u_{i,j}^s \in \{0,1\}$ & $1$ if customer~$j$ is served from WH~$i$, $0$ otherwise
\end{tabular}
\def\arraystretch{1}

\medskip

Whenever we write one of the variables without an index, we mean the vector of all corresponding variables, e.g., $w^0\coloneqq (w_i^0)_{i\in I}$ and $r^s=(r^s_{v_1,v_2})_{v_1,v_2\in V}$.
In the general problem formulation~(2-RO), the vector of first-stage decision variables is then $x=(w^0, a^0)$ and the vector of second-stage decision variables for scenario~$s\in S$ is $y^s=(w^s, a^s, r^s, t^s, u^s)$.

\medskip

The value of the first-stage objective~$f$ is given by the sum of the fixed and the size-dependent costs of opening warehouses in the first stage:
\begin{align*}
    f(w^0,  a^0) &\coloneqq \sum_{i \in I} e_i\cdot w^0_i +  \sum_{i \in I} d_i \cdot a_i^0.
\end{align*}

Recoverability is modeled by the second-stage decision variables~$w^s_i$ and~$a^s_i$ for each warehouse~$i\in I$ in each scenario~$s\in S$. These reflect the actual choice of opened warehouses and warehouse sizes to be implemented if scenario~$s$ is realized. Note that we do not allow warehouses to be closed or decreased in size in the second stage. This is a reasonable assumption in supply chain modeling since downsizing or closing warehouses mid-horizon generally leads to high operational and transition costs without clear financial benefits~\cite{hubner2013demand}.
If warehouse~$i$ is not opened in the first stage ($w^0_i = 0$) but is opened in the second stage when scenario~$s$ realizes ($w^s_i = 1$), a penalty of $e'_i - e_i>0$ has to be paid in addition to the opening cost. Similarly, for the size-dependent cost, if the size of warehouse~$i$ is increased in the second stage in scenario~$s$ (i.e., if $a_i^s-a_i^0 > 0$), a penalty of $d'_i - d_i>0$ has to be paid per unit of size in addition to the cost in the first stage.

\noindent
The second-stage objective~$g^s$ for a scenario~$s \in S$ is given as:
\begin{align}
	g^s(w^0,  a^0, w^s, a^s, r^s, t^s, u^s)
	\;\coloneqq\; &
	\sum_{i  \in I} e'_i\cdot (w^s_i-w^0_i) + \sum_{i \in I} d'_i\cdot (a_i^s - a_i^0)  \label{eq:ssorec}
	\\ & + \sum_{v_1,v_2\in V} c_{v_1,v_2}\cdot r^s_{v_1,v_2}   \label{eq:ssodist}
	\\ & + \sum_{v_1,v_2\in V} \alpha_{v_1,v_2}\cdot r^s_{v_1,v_2} \refstepcounter{equation}\subeqn \label{eq:ssoem1}
	\\ & + \sum_{v_1,v_2\in V} \gamma_{v_1,v_2}\cdot t^s_{v_1,v_2}  \subeqn \label{eq:ssoem2}
	\\ & + \sum_{i\in I, j\in J} F\cdot r^s_{i,j}. \label{eq:ssoroute}
\end{align}

The first term~\eqref{eq:ssorec} describes the recovery costs resulting from opening additional warehouses and increasing warehouse sizes, and~\eqref{eq:ssodist} corresponds to the travel costs of the tours. The terms~\eqref{eq:ssoem1} and~\eqref{eq:ssoem2} represent the cost of emissions of the vehicles and of the cargo load, respectively, while~\eqref{eq:ssoroute} accounts for the fixed cost of the used vehicles, which are obtained by multiplying the fixed cost~$F$ per vehicle by the number of used arcs leaving the warehouses (i.e., by the number of tours).

\noindent
The feasible set of the first-stage variables is given as
\begin{align*}
\mathcal{X} \coloneqq \{(w^0,  a^0) \in \{0,1\}^I \times \R_{\ge 0}^I : a_i^0 \le A_i \cdot w^0_i \text{ for all } i \in I\},
\end{align*}
which forces the first-stage size of a warehouse to zero if the warehouse is not opened in the first stage, and upper bounds it by the corresponding maximum warehouse capacity if it is opened in the first stage. Given a scenario~$s \in S$ and a first-stage solution~$x=(w^0, a^0)$, the feasible set~$\mathcal{Y}^s(x)$ of the second-stage variables is described by the following constraints:

\begingroup
\allowdisplaybreaks
\addtolength{\jot}{0.5em}
\begin{align}
	& \sum_{v\in V \setminus \{j\}} r^s_{v,j} = \sum_{v\in V \setminus \{j\}} r^s_{j,v} = 1, && \forall j \in J: \beta^s_j > 0 \label{eq:m-5}\\
        & \sum_{v\in V \setminus \{j\}} r^s_{v,j} = \sum_{v\in V \setminus \{j\}} r^s_{j,v} = 0, && \forall j \in J: \beta^s_j = 0 \label{eq:m-5b}\\
	& \sum_{j\in J} r^s_{i,j} = \sum_{j\in J} r^s_{j,i},  && \forall i \in I \label{eq:m-6}\\
	& \sum_{v\in V \setminus \{j\}} t^s_{v,j} = \sum_{v\in V \setminus \{j\}} t^s_{j,v} + \beta^s_j &&\forall j \in J, \label{eq:m-7}\\
	& t^s_{v_1,v_2} \le L \cdot r^s_{v_1,v_2},&&\forall v_1,v_2 \in V \label{eq:m-8}\\
	& \sum_{j\in J} t^s_{i,j} \le a^s_i ,  && \forall i\in I \label{eq:m-9}\\
	& a^0_i \le a^s_i \le A_i \cdot w^s_i, && \forall i\in I \label{eq:m-10}\\
	& w^0_i \le w^s_i, && \forall i\in I \label{eq:m-11}\\
	& u^s_{i,j_1} - u^s_{i,j_2} \le 1 - r^s_{j_1,j_2} - r^s_{j_2,j_1},  && \forall i\in I, j_1,j_2 \in J, \; j_1 \ne j_2 \label{eq:m-12} \\
	& u^s_{i,j} \ge r^s_{i,j}, && \forall i\in I, j\in J \label{eq:m-13}\\
	& u^s_{i,j} \ge r^s_{j,i}, && \forall i\in I, j\in J \label{eq:m-14}\\
	& \sum_{i\in I} u^s_{i,j} = 1, && \forall j \in J: \beta^s_j > 0 \label{eq:m-15}\\
        & w^s\in\{0,1\}^{I}, \;a^s\in\R_{\ge 0}^{I}\\
        & \rlap{ $r^s\in\{0,1\}^{V \times V}, \;t^s\in\R_{\ge 0}^{V \times V}, \;u^s\in\{0,1\}^{I \times J}$. } \label{eq:m-end}
\end{align}
\endgroup

Constraints~\eqref{eq:m-5}--\eqref{eq:m-6} fix the routing in the network by setting the number of incoming and outgoing used arcs of each active customer node to~1, and guaranteeing that each warehouse has as many outgoing as incoming used arcs. \rev{Customers with zero demand should not be visited by any route.} Constraints~\eqref{eq:m-7}--\eqref{eq:m-9} describe the flow conservation in the network and fix the flow on inactive arcs to~$0$. In~\eqref{eq:m-10}, we consider the capacity constraints of warehouses and ensure that warehouse sizes cannot be decreased in the second stage. Similarly, by~\eqref{eq:m-11}, warehouses opened in the first stage cannot be closed anymore. The last four constraints are needed to restrict the routing to closed tours, i.e., each tour starts and ends at the same warehouse.
This is modeled using the auxiliary variables~$u^s$ that describe which warehouse serves which customers.
Constraint~\eqref{eq:m-15} ensures that each customer with non-zero demand is served from exactly one warehouse.
If two customer nodes are connected, they have to be served from the same warehouse by~\eqref{eq:m-12}.
Together with~\eqref{eq:m-13} and~\eqref{eq:m-14}, this ensures that the first and the last arc of a tour are connected to the same warehouse.

\medskip
\noindent
The full model, reformulated as a minimization problem, then reads:

\begin{align}\tag{RCLRP}
    \begin{aligned}
        \min \quad&f(w^0,  a^0) + z \\[1em]
        \text{s.t.}\quad&a_i^0 \le A_i w^0_i \quad && \quad\forall i \in I\\[0.8em]
        &z \ge g^s(w^0,  a^0, w^s, a^s, r^s, t^s, u^s) && \quad\forall s \in S \\[0.8em]
        &\eqref{eq:m-5} - \eqref{eq:m-end} && \quad\forall s \in S\\[0.8em]
        & z \in \R, \;w^0 \in \{0,1\}^I, \;a^0 \in \mathbb{R}_{\geq 0}^I.
    \end{aligned}
\end{align}
\vspace{1em}

 \section{BACASP Model Formulation}\label{sec:appendix-3-BACASP}

 For completeness, we state the full BACASP model from~\cite{RODRIGUES2021499} here. For a more detailed description, we refer to the original paper~\cite{RODRIGUES2021499}. An instance of BACASP is given by a set $V = \{1, \dots, N\}$ of $N$~vessels, a set $T = \{1, \dots, M\}$ of $M$~time periods, and a set $G = \{1, \dots, C\}$ of $C$~cranes. Further, the wharf is divided into $J+1$ berth sections, numbered with indices in $B = \{0, \dots, J\}$. The required parameters and variables are given as follows:

\bigskip
\noindent
\textbf{Parameters:}

\noindent
\def\arraystretch{1.2}
\begin{tabular}{wc{0.18\linewidth}|p{0.76\linewidth}}
	$H_k$ & length of vessel~$k$ (number of covered berth sections) \\
	$Q_k$ & cargo volume loaded on vessel~$k$ \\
	$A^s$ & arrival times of vessels in scenario $s$ \\
	$NC_k$ & maximum number of cranes working simultaneously on vessel~$k$ \\
	$F$ & safety time between vessel departures and berthing \\
	$S_g$ & crane $g$ can operate starting at berth section~$S_g$ \\
        $E_g$ & crane~$g$ can operate up to berth section~$E_g$ \\
	$P_g$ & processing rate of crane~$g$ \\
\end{tabular}
\def\arraystretch{1}

\bigskip

\noindent
\textbf{First-stage decision variables:}

\noindent
\def\arraystretch{1.2}
\begin{tabular}{wc{0.18\linewidth}|p{0.76\linewidth}}
    $e_{k,\ell} \in \{0,1\}$ & if $e_{k,\ell}$ is equal to~$1$, then vessel~$\ell$ starts to be served after vessel~$k$ departs \\
	$u_{k,\ell} \in \{0,1\}$ & $1$ if vessel~$\ell$ berths completely below the berth position of vessel~$k$, $0$~otherwise \\
	$b_k \in B$ & berthing position of vessel~$k$ \\
	$\pi_{k,n} \in \{0,1\}$ & $1$ if vessel~$k$ starts at berth position~$n$, $0$~otherwise \\
	$\sigma_{k,n} \in \{0,1\}$ & $1$ if berth section~$n$ is assigned to vessel~$k$, $0$~otherwise
\end{tabular}
\def\arraystretch{1}

\bigskip
\newpage
\noindent
\textbf{Second-stage decision variables for scenario~$s \in S$:}

\noindent
\begin{tabular}{wc{0.18\linewidth}|p{0.76\linewidth}}
    $d^s_{g,k,j} \in \{0,1\}$ & $1$ if crane~$g$ is assigned to vessel~$k$ in period~$j$ in scenario~${s \in S}$, $0$~otherwise \\
	$t^s_k \in T$ & berthing time of vessel~$k$ in scenario~$s \in S$ \\
	$c^s_k \in T$ & departure time of vessel~$k$ in scenario~$s \in S$ \\
	$\alpha^s_{k,j} \in \{0,1\}$ & $1$ if vessel~$k$ starts operating in period~$j$ in scenario~${s \in S}$, $0$~otherwise \\
	$\beta^s_{k,j} \in \{0,1\}$ & $1$ if vessel~$k$ is operating in period~$j$ in scenario~$s \in S$, $0$~otherwise \\
	$\gamma^s_{k,j} \in \{0,1\}$ & $1$ if the last operating period of vessel~$k$ is~$j$ in scenario~${s \in S}$, $0$~otherwise
\end{tabular}

\bigskip
As before, whenever we write one of the variables without an index, we mean the vector of all corresponding variables, e.g., $e\coloneqq (e_{k,\ell})_{k,\ell\in V}$ or $c^s \coloneqq (c^s)_{k\in V}$. The vector $x$ of first-stage decision variables is then $x=(e,u,b,\pi,\sigma)$ and the vector $y^s$ of second-stage decision variables for scenario~$s\in S$ is $y^s=(d^s,t^s,c^s,\alpha^s,\beta^s,\gamma^s)$.

\bigskip

\noindent
The first-stage objective is given as $f(e,u,b,\pi,\sigma)\coloneqq 0$ and the second-stage objective for a scenario~$s \in S$ is given as follows (for simplicity, we only list the relevant variables as arguments):
\begin{align}
	g^s(c^s)
	\coloneqq \sum_{k \in V} (c^s_k - A^s_k).
\end{align}

\noindent
The first-stage constraints are:
\allowdisplaybreaks
\begin{align}
	&e_{\ell,k} + e_{k,\ell} + u_{\ell,k} + u_{k,\ell} = 1, && \forall k, \ell \in V, \, k < \ell \label{eq:baqstart1} \\
	&b_k \leq J - H_k + 1, && \forall k \in V \\
	&b_k \geq b_{\ell} + H_{\ell} + (u_{k,\ell} - 1)\cdot (J+1), && \forall k, \ell \in V, \, k \neq \ell \\
	&b_k \le b_\ell+H_\ell-1+u_{k,\ell}\cdot J, && \forall k, \ell \in V, \, k \ne \ell \\
	&b_k \in \N_0, && \forall k \in V \\
	&e_{k,\ell}, u_{k,\ell} \in \{0,1\}, && \forall k \in V, j \in T \\
	&b_k = \sum_{n \in B} n \cdot \pi_{k,n}, && \forall k \in V \\
	&b_k, t_k, c_k \in \mathbb{Z}_0^+, && \forall k \in V \\
	&\sum_{n \in B} \sigma_{k,n} = H_k, && \forall k \in V \\
	&\sum_{n \in B} \pi_{k,n} = 1, && \forall k \in V \\
	&\pi_{k,n} \geq \sigma_{k,n} - \sigma_{k,n-1}, && \forall k \in V, \, n \in B, \, n > 0 \\
	&\pi_{k,0} \geq \sigma_{k,0}, && \forall k \in V \\
	&\pi_{k,n} \leq \sigma_{k,n}, && \forall k \in V, \, n \in B \\
	&\pi_{k,n} \leq 1 - \sigma_{k,n-1}, && \forall k \in V, \, n \in B, \, n > 0 \\
	&u_{k,\ell} + \sum^J_{m=\max\{n - H_{\ell} + 1, 0\}} \pi_{\ell,m} + \pi_{k,n} \leq 2, && \forall k, \ell \in V, \, k \neq \ell, \, n \in B \\
        & e_{k,l}, u_{k,l}, b_k, \pi_{k,n}, \sigma_{k,n} && \forall k, l \in V, n \in B. \label{eq:baqend1}
\end{align}

\bigskip
\noindent
Given a scenario~$s \in S$, and a first-stage solution~$(e,u,b,\pi,\sigma)$, the constraints of the second stage are:
\allowdisplaybreaks
\begin{align}
	&t^s_\ell \geq c^s_k + F + (e_{k,\ell} - 1)\cdot (M+F), && \forall k, \ell \in V, \, k \neq \ell \label{eq:baqstart2}\\
	&t^s_k \geq A^s_k, && \forall k \in V \\
	&\sum_{k\in V} d^s_{g,k,j} \leq 1, && \forall j \in T, \, g \in G \\
	&t^s_k \leq j\cdot d^s_{g,k,j} + (1 - d^s_{g,k,j})\cdot M, && \forall j \in T, \, k \in V, \, g \in G \\
	&c^s_k \geq (j + 1)\cdot  d^s_{g,k,j}, && \forall j \in T, \, k \in V, \, g \in G \\
	&\sum_{j \in T} \sum_{g \in G} P_g \cdot d^s_{g,k,j} \geq Q_k, && \forall k \in V \\
	&b_k + H_k \leq E_g \cdot d^s_{g,k,j} + (1 - d^s_{g,k,j})\cdot (J+1), && \forall j \in T, \, k \in V, \, g \in G \\
	&b_k \geq S_g \cdot d^s_{g,k,j}, && \forall j \in T, \, k \in V, \, g \in G \\
	&d^s_{g,k,j} + d^s_{g',\ell,j} \leq 2 - u_{k,\ell},  && \hspace*{-22mm}\forall j \in T, \, k, \ell \in V,\,g, g' \in G, \, g' < g \\
        &\sum_{g \in G} d^s_{g,k,j} \leq NC_k, && \forall k \in V, \, j \in T \\
	&b_k, t^s_k, c^s_k \in \N_0, && \forall k \in V \\
	&d^s_{g,k,j} \leq \sum^{E_g - H_k}_{n=S_g} \pi_{k,n}, && \forall k \in V, \, g \in G, \, j \in T \\
	&t^s_k = \sum_{j \in T} j \cdot \alpha^s_{k,j}, && \forall k \in V \\
	&c^s_k \geq (j + 1) \cdot \beta^s_{k,j}, && \forall k \in V, \, j \in T \\
	&d^s_{g,k,j} \leq \beta^s_{k,j}, && \forall k \in V, \, j \in T, \, g \in G \\
	&\sum_{j\in T} \alpha^s_{k,j} = 1, && \forall k \in V \\
	&\alpha^s_{k,j} \geq \beta^s_{k,j} - \beta^s_{k, j-1}, && \forall k \in V, \, j \in T, \, j > 1 \\
	&\alpha^s_{k,1} \geq \beta^s_{k,1}, && \forall k \in V, \\
	&\alpha^s_{k,j} \leq \beta^s_{k,j}, && \forall k \in V, \, j \in T \\
	&\alpha^s_{k,j} \leq 1 - \beta^s_{k,j-1}, && \forall k \in V, \, j \in T, \, j > 1 \\
	&e_{k,\ell} + \beta^s_{k,i} + \alpha^s_{\ell,j} \leq 2, && \hspace*{-25mm}\forall k, \ell  \in V, \, k \neq \ell , \, j, i \in T, \, i \geq j - F \\
	&\gamma^s_{k,j} \geq \beta^s_{k,j} - \beta^s_{k,j+1}, && \forall k \in V, \, j \in T, \, j < M \\
	&\gamma^s_{k,M} \geq \beta^s_{k,M}, && \forall k \in V, \\
	&\gamma^s_{k,j} \leq \beta^s_{k,j}, && \forall k \in V, \, j \in T \\
	&\gamma^s_{k,j} \leq 1 - \beta^s_{k,j+1}, && \forall k \in V, \, j \in T, \, j < M \\
	&\sum_{j \in T} \gamma^s_{k,j} = 1, && \forall k \in V \\
	&d^s_{g,k,j}, \alpha^s_{k,j}, \beta^s_{k,j}, \gamma^s_{k,j} \in \{0, 1\}, && \forall g \in G, \, k \in V, \, j \in T. \label{eq:baqend2}
\end{align}

\medskip
\noindent
The full BACASP model then reads:
\begin{align}\label{eq:bacasp}\tag{BACASP}
\begin{aligned}
	\begin{aligned}
		\min \quad&z \\
		\text{s.t.}\quad &\eqref{eq:baqstart1} - \eqref{eq:baqend1}  \\
		&z \ge g^s(c^s) && \quad\forall s \in S \\
		&\eqref{eq:baqstart2} - \eqref{eq:baqend2} && \quad\forall s \in S.
	\end{aligned}
\end{aligned}
\end{align}

 \section{Additional experiments}\label{sec:appendix-4-results}

\revv{In this section, we present additional experiments comparing different master-gap factors (\ref{sec:appendix-mgap}) and implementations of \Init() in our ASBP (\ref{sec:appendix-init}), as well as experiments assessing the effect of disabling the heuristic (\ref{sec:appendix-noheur}).}

\subsection{Comparison of different master-gap factors}\label{sec:appendix-mgap}
\rev{
In addition to the master-gap factor~$\mu=0.5$ that was used throughout the experiments in Section~6 of the main paper, we tested several other values for the master-gap factor~$\mu$ \rev{in the ASBP}. For both considered applications (RCLRP and BACASP), we evaluated the performance on small instances. The results are reported in Figure~\ref{fig:mgapfactor} for the RCLRP and the BACASP \revv{and target gaps~$P$ of $\{5\%,\, 10\%\}$.\footnote{No experiments were conducted for a target gap of zero since the master-gap factor~$\mu$ is irrelevant in this case.}}
}

\begin{figure}[h]
\includegraphics{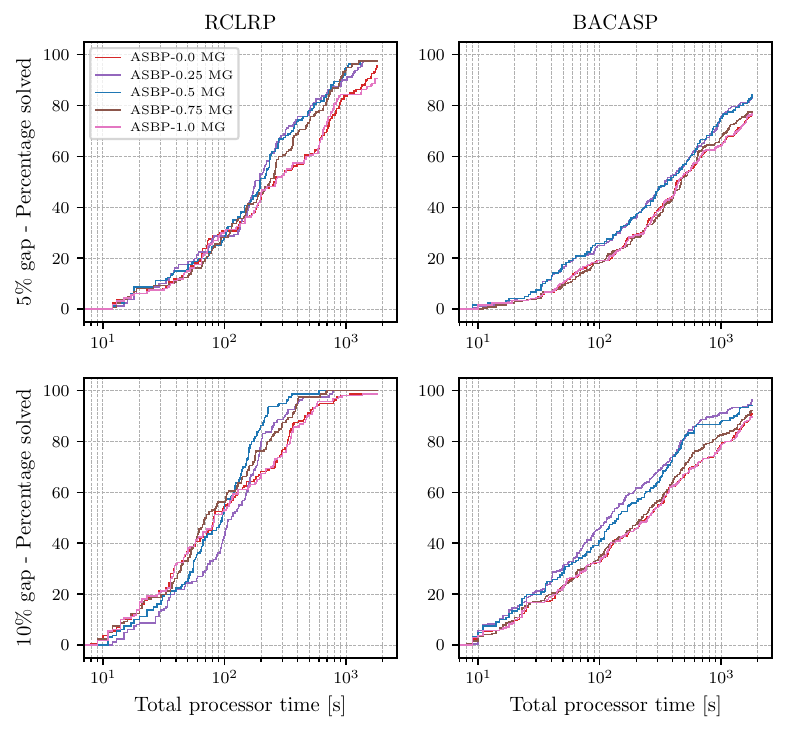}
\caption{Performance plots of our ASBP with different master-gap factors~$\mu$.}
\label{fig:mgapfactor}
\end{figure}

\revv{For the RCLRP, it can be seen that certain master-gap factors are good choices for easier instances (instances that are solved in less time), and some factors perform better for harder instances. For example, in the experiments with a target gap of 10\%, $\mu=0.25$ produces the worst results out of all tested master-gap factors for instances that solve within \SI{200}{s}, while this factor is the second best choice for instances that take more than \SI{200}{s} to solve. For the experiments with target gap 5\%, runs with $\mu=0.25$ and with $\mu=0.5$ perform similarly well. Overall, we identified $\mu=0.5$ to be the best master-gap factor choice, and choose it consistently for both applications and each setup of parameters (except experiments with target gap zero where the choice of the master-gap factor is irrelevant).}

\subsection{Comparison of different implementations of \Init() for the \linebreak RCLRP}\label{sec:appendix-init}
\rev{
As mentioned in Section~3 of the main paper, several ways to initialize the subset~$D$ of scenarios are possible. To compare different options, we ran additional ASBP variants using the following \Init() options:
\begin{itemize}
    \item \textbf{Empty:} $\Init()=\emptyset$.
    \item \textbf{Random:} $\Init()=\{s\}$, where $s \in S$ is a randomly chosen scenario.
    \item \textbf{Demand:} $\Init()=\underset{s\in S}{\operatorname{lexmax}} \left(\lvert \{j\in J:\beta^s_j > 0\} \rvert,\; \sum_{j\in J} \beta^s_j\right)$.
\end{itemize}
}

\rev{
The last option chooses the scenario with the maximum number of customers with non-zero demand, breaking ties by choosing the scenario with the highest total demand. This is motivated by the fact that, in the RCLRP, scenarios with more customers and higher total demand are likely to be more challenging.
}

\revv{The results are reported in Table~\ref{tab:init-table}. One can see that the choice of \Init() does not significantly affect the overall average runtime, which ranges between \SI{1946}{s} for the ``Empty'' option and \SI{1965}{s} for the ``Demand'' option.
As for the number of runs that were solved within the time limit, the ``Demand'' option is slightly better than the other two. Especially for some configurations of runs with 0\% target gap, using the ``Demand'' option, we were able to solve more instances than with the other two options (15.5 versus 15.0 and 15.1). For these, however, the average running times are very high and often close to the time limit anyways, which makes this result less significant with respect to the performance of this option.
Because of the similar performances, no clear winner could be identified, and we chose to use the ``Empty'' option in the main computational results in order to prove the concept of having the algorithm itself identify a first scenario to add to the subset $D$.}

\begin{center}
\begin{table}[H]
    \centering
    \begin{adjustbox}{width=\textwidth}
\newcolumntype{R}[1]{>{\raggedleft\arraybackslash}m{#1}}
\begin{tabular}{rrr|*{3}{R{18mm}}|*{3}{R{20mm}}}
\toprule
    &    &    &  \multicolumn{3}{c}{Average total runtime} & \multicolumn{3}{|c}{Number of solved runs within time limit} \\
Gap & Cust. & Sc. &  Empty &  Random &  Demand &  Empty &  Random &  Demand \\
\midrule
0.00 & 12 & 16 &    134 &     166 &        182 &   20 &   20 &    20 \\
     &    & 64 &    866 &     831 &        773 &   18 &   18 &    18 \\
     & 18 & 16 &   4385 &    3347 &       3474 &    8 &   10 &    10 \\
     &    & 64 &   6000 &    6000 &       5491 &    0 &    0 &     2 \\
     & 22 & 16 &   5232 &    5288 &       5294 &    4 &    4 &     4 \\
     &    & 64 &   6000 &    6000 &       6000 &    0 &    0 &     0 \\
     & 24 & 16 &   4808 &    5228 &       5506 &    7 &    6 &     6 \\
     &    & 64 &   5703 &    5725 &       5326 &    2 &    4 &     6 \\
\midrule
\multicolumn{3}{c|}{Average 0\% gap}     &   4141 &    4073 &       4006 &   7.4 &   7.8 &   8.2 \\
\midrule
0.05 & 12 & 16 &     53 &      81 &         82 &   20 &   20 &    20 \\
     &    & 64 &    187 &     189 &        137 &   20 &   20 &    20 \\
     & 18 & 16 &    289 &     426 &        790 &   20 &   20 &    19 \\
     &    & 64 &    419 &     923 &        472 &   20 &   20 &    20 \\
     & 22 & 16 &   1650 &    1284 &       1543 &   16 &   18 &    18 \\
     &    & 64 &   2134 &    2878 &        942 &   18 &   14 &    20 \\
     & 24 & 16 &   1627 &    1786 &       3128 &   18 &   16 &    14 \\
     &    & 64 &   2707 &    1604 &       2481 &   17 &   20 &    18 \\
\midrule
\multicolumn{3}{c|}{Average 5\% gap}    &   1133 &    1146 &       1197 &  18.6 &  18.5 &  18.6 \\
\midrule
0.10 & 12 & 16 &     27 &      24 &         39 &   20 &   20 &    20 \\
     &    & 64 &     81 &     101 &         77 &   20 &   20 &    20 \\
     & 18 & 16 &    124 &     214 &        219 &   20 &   20 &    20 \\
     &    & 64 &    199 &     521 &        352 &   20 &   20 &    20 \\
     & 22 & 16 &    825 &     660 &       1116 &   18 &   19 &    18 \\
     &    & 64 &   1318 &    2112 &        742 &   18 &   14 &    20 \\
     & 24 & 16 &    808 &     931 &       1748 &   18 &   20 &    18 \\
     &    & 64 &   1134 &     500 &       1237 &   18 &   20 &    20 \\
\midrule
\multicolumn{3}{c|}{Average 10\% gap}     &    565 &     633 &        691 &  19.0 &  19.1 &  19.5 \\
\midrule
\multicolumn{3}{c|}{Overall average} &   1946 &    1951 &       1965 &  15.0 &  15.1 &  15.5 \\
\bottomrule
\end{tabular}
\end{adjustbox}

    \caption{Comparison of \Init() variants for the ASBP (with $\mu=0.5$).
    Each provided average total runtime is the average over 10~instances and 2~repetitions per instance (i.e., 20~data points are used for each reported average). The time limit is set to 100~minutes. All times are in seconds.}
    \label{tab:init-table}
\end{table}
\end{center}

\subsection{Effect of disabling the heuristic}\label{sec:appendix-noheur}
\revv{
In Figure~\ref{fig:noheuristic}, we examine how the performance of the three considered methods changes when the heuristic is disabled. In the algorithms, this effectively leads to the heuristic phase being skipped, and all scenarios' upper bounds being initialized with~$+\infty$ (and lower bounds with~$-\infty$) before going into the first iteration of the master problem.}

\begin{figure}[H]
\includegraphics{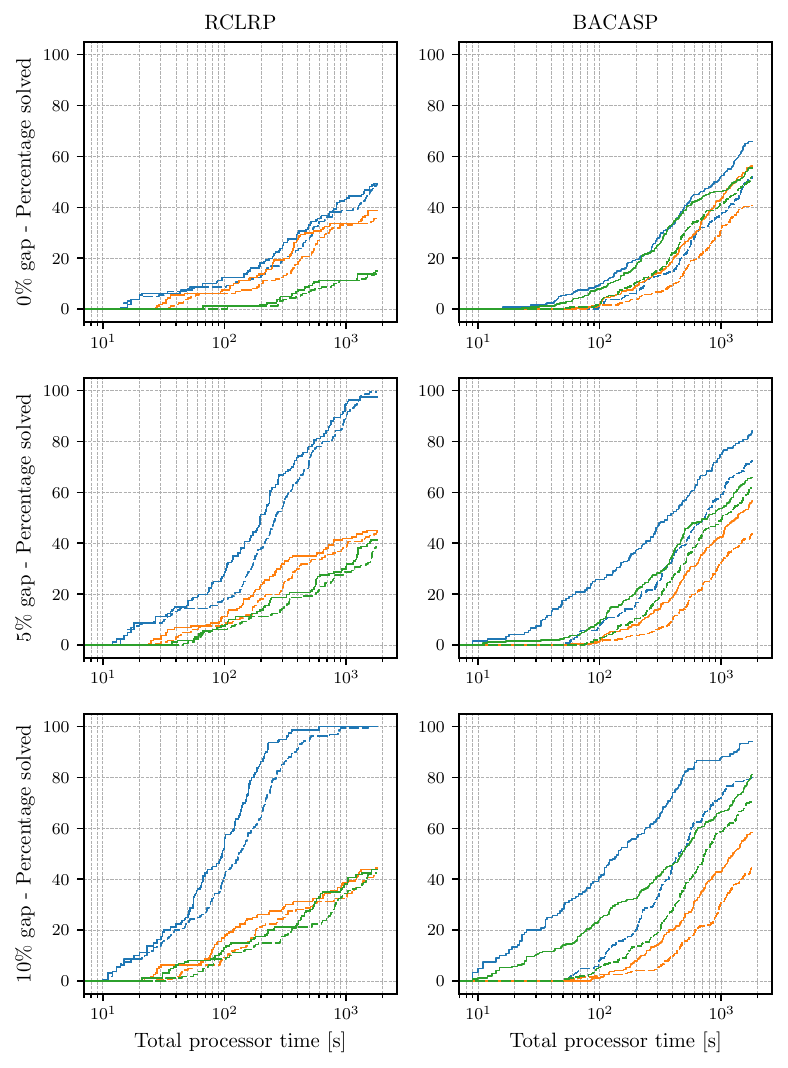}
\caption{\rev{Comparison of performance plots for different algorithms (ASBP in blue, ISAM in orange, SRP in green) shown once with the standard heuristic (solid lines) and once without heuristic (dashed lines). All other parameters are the same as for the small instances in Figures~1 and~2 of the main paper.}}
\label{fig:noheuristic}
\end{figure}

\revv{Throughout these results, it is apparent that the use of a heuristic does indeed speed up the overall solving process on average. This is true for each tested algorithm, each target gap, and each problem.
Recall that, for the RCLRP, we used a \SI{0.1}{s}-run of Gurobi as the heuristic while, for the BACASP, a specifically tailored combinatorial heuristic is used. This also shows in the results: for the RCLRP, the difference between the performance with and without the use of the heuristic is considerably smaller than for the BACASP. In the latter case, e.g., in the experiments with 10\% target gap, 94.2\% of the instances were solved within the time limit by our algorithm using the heuristic, and only 80\% were solved within the time limit without the use of the heuristic, which makes a difference of roughly 14\%.}

\fi


\begin{thebibliography}{MTLN{\etalchar{+}}15}

\bibitem[ABV09]{aissi2009min}
H.~Aissi, C.~Bazgan, and D.~Vanderpooten.
\newblock Min--max and min--max regret versions of combinatorial optimization
  problems: A survey.
\newblock {\em European Journal of Operational Research}, 197(2):427--438,
  2009.

\bibitem[AMFL15]{alvarez}
E.~\'Alvarez-Miranda, E.~Fern\'andez, and I.~Ljubi\'c.
\newblock The recoverable robust facility location problem.
\newblock {\em Transportation Research Part B}, 79:93--120, 2015.

\bibitem[AO18]{AGRA2018138}
A.~Agra and M.~Oliveira.
\newblock {MIP} approaches for the integrated berth allocation and quay crane
  assignment and scheduling problem.
\newblock {\em European Journal of Operational Research}, 264(1):138--148,
  2018.

\bibitem[BK24]{bib13bertsimas2024machine}
D.~Bertsimas and C.~W. Kim.
\newblock A machine learning approach to two-stage adaptive robust
  optimization.
\newblock {\em European Journal of Operational Research}, 319(1):16--30, 2024.

\bibitem[B{\"O}08]{bib29bienstock2008computing}
D.~Bienstock and N.~{\"O}zbay.
\newblock Computing robust basestock levels.
\newblock {\em Discrete Optimization}, 5(2):389--414, 2008.

\bibitem[BTGGN04]{bib27ben2004adjustable}
A.~Ben-Tal, A.~Goryashko, E.~Guslitzer, and A.~Nemirovski.
\newblock Adjustable robust solutions of uncertain linear programs.
\newblock {\em Mathematical Programming}, 99(2):351--376, 2004.

\bibitem[BTNE09]{ben2009robust}
A.~Ben-Tal, A.~Nemirovski, and L.~{El Ghaoui}.
\newblock {\em Robust Optimization}.
\newblock Princeton University Press, 2009.

\bibitem[CCG13]{contardo2013computational}
C.~Contardo, J.F. Cordeau, and B.~Gendron.
\newblock A computational comparison of flow formulations for the capacitated
  location-routing problem.
\newblock {\em Discrete Optimization}, 10(4):263--295, 2013.

\bibitem[CZS23]{bib18chargui2023berth}
K.~Chargui, T.~Zouadi, and V.~R. Sreedharan.
\newblock Berth and quay crane allocation and scheduling problem with renewable
  energy uncertainty: A robust exact decomposition.
\newblock {\em Computers \& Operations Research}, 156:106251, 2023.

\bibitem[GH24]{goerigk2024buch}
M.~Goerigk and M.~Hartisch.
\newblock {\em An Introduction to Robust Combinatorial Optimization}, volume
  361 of {\em International Series in Operations Research \& Management
  Science}.
\newblock Springer, 2024.

\bibitem[GK24]{goerigk2024data}
M.~Goerigk and J.~Kurtz.
\newblock Data-driven prediction of relevant scenarios for robust combinatorial
  optimization.
\newblock {\em Computers \& Operations Research}, 174:106886, 2024.

\bibitem[GS14]{bib6goerigk2014recovery}
M.~Goerigk and A.~Sch{\"o}bel.
\newblock Recovery-to-optimality: A new two-stage approach to robustness with
  an application to aperiodic timetabling.
\newblock {\em Computers \& Operations Research}, 52:1--15, 2014.

\bibitem[{Gur}25]{MIP-gap-gurobi}
{Gurobi Optimization}.
\newblock {What is the MIPGap?}
\newblock
  \url{https://support.gurobi.com/hc/en-us/articles/8265539575953-What-is-the-MIPGap},
  2025.
\newblock Accessed on October 10, 2025.

\bibitem[HJPR20]{hashemi}
H.~{Hashemi Doulabi}, P.~Jaillet, G.~Pesant, and L.-M. Rousseau.
\newblock Exploiting the structure of two-stage robust optimization models with
  exponential scenarios.
\newblock {\em INFORMS Journal on Computing}, 33(1):143--162, 2020.

\bibitem[HKS13]{hubner2013demand}
A.~H. H{\"u}bner, H.~Kuhn, and M.~G. Sternbeck.
\newblock Demand and supply chain planning in grocery retail: an operations
  planning framework.
\newblock {\em International Journal of Retail \& Distribution Management},
  41(7):512--530, 2013.

\bibitem[HLLU10]{hendriks2010robust}
M.~Hendriks, M.~Laumanns, E.~Lefeber, and J.~T. Udding.
\newblock Robust cyclic berth planning of container vessels.
\newblock {\em OR Spectrum}, 32:501--517, 2010.

\bibitem[LLMS09]{bib4liebchen2009concept}
C.~Liebchen, M.~L{\"u}bbecke, R.~H. M{\"o}hring, and S.~Stiller.
\newblock The concept of recoverable robustness, linear programming recovery,
  and railway applications.
\newblock In R.~K. Ahuja, R.~H. M{\"o}hring, and C.~D. Zaroliagis, editors,
  {\em Robust and Online Large-Scale Optimization: Models and Techniques for
  Transportation Systems}, volume 5868 of {\em Lecture Notes in Computer
  Science}, pages 1--27. Springer, 2009.

\bibitem[LLZW24]{bib16li2024decomposition}
Y.~Li, X.~Li, C.~Zhang, and T.~Wu.
\newblock Decomposition algorithms for the robust unidirectional quay crane
  scheduling problems.
\newblock {\em Computers \& Operations Research}, 167:106670, 2024.

\bibitem[LR57]{luceraiffabook}
R.~D. Luce and H.~Raiffa.
\newblock {\em Games and Decisions: Introduction \& Critical Surevey}.
\newblock Wiley, 1957.

\bibitem[LXZ20]{liu2020two}
C.~Liu, X.~Xiang, and L.~Zheng.
\newblock A two-stage robust optimization approach for the berth allocation
  problem under uncertainty.
\newblock {\em Flexible Services and Manufacturing Journal}, 32:425--452, 2020.

\bibitem[MTLN{\etalchar{+}}15]{montoya2015literature}
J.~R. Montoya-Torres, J.~{López Franco}, S.~{Nieto Isaza}, H.~{Felizzola
  Jiménez}, and N.~Herazo-Padilla.
\newblock A literature review on the vehicle routing problem with multiple
  depots.
\newblock {\em Computers \& Industrial Engineering}, 79:115--129, 2015.

\bibitem[PP14]{PRODHON20141}
C.~Prodhon and C.~Prins.
\newblock A survey of recent research on location-routing problems.
\newblock {\em European Journal of Operational Research}, 238(1):1--17, 2014.

\bibitem[RA21]{RODRIGUES2021499}
F.~Rodrigues and A.~Agra.
\newblock An exact robust approach for the integrated berth allocation and quay
  crane scheduling problem under uncertain arrival times.
\newblock {\em European Journal of Operational Research}, 295(2):499--516,
  2021.

\bibitem[TA18]{toenissen11}
D.~D. T\"onissen and J.~J. Arts.
\newblock Economies of scale in recoverable robust maintenance location routing
  for rolling stock.
\newblock {\em Transportation Research Part B}, 117:360--377, 2018.

\bibitem[TAS19]{toenissen12}
D.~D. T\"onissen, J.~J. Arts, and Z.-J.~M. Shen.
\newblock Maintenance location routing for rolling stock under line and fleet
  planning uncertainty.
\newblock {\em Transportation Science}, 53(5):1252--1270, 2019.

\bibitem[TFEG17]{toro}
E.~M. Toro, J.~F. Franco, M.~G. Echeverri, and F.~G. Guimarães.
\newblock A multi-objective model for the green capacitated location-routing
  problem considering environmental impact.
\newblock {\em Computers \& Industrial Engineering}, 110:114--125, 2017.

\bibitem[TSC23]{TSANG202392}
M.~Y. Tsang, K.~S. Shehadeh, and F.~E. Curtis.
\newblock An inexact column-and-constraint generation method to solve two-stage
  robust optimization problems.
\newblock {\em Operations Research Letters}, 51(1):92--98, 2023.

\bibitem[WMZ{\etalchar{+}}24]{bib17wang2024robust}
C.~Wang, L.~Miao, C.~Zhang, T.~Wu, and Z.~Liang.
\newblock Robust optimization for the integrated berth allocation and quay
  crane assignment problem.
\newblock {\em Naval Research Logistics}, 71(3):452--476, 2024.

\bibitem[YGd19]{YANIKOGLU2019799}
{\.I}.~Yan{\i}ko{\u{g}}lu, B.~L. Gorissen, and D.~{den Hertog}.
\newblock A survey of adjustable robust optimization.
\newblock {\em European Journal of Operational Research}, 277(3):799--813,
  2019.

\bibitem[ZZ13]{zeng}
B.~Zeng and L.~Zhao.
\newblock Solving two-stage robust optimization problems using a
  column-and-constraint generation method.
\newblock {\em Operations Research Letters}, 41(5):457--461, 2013.

\end{thebibliography}
\end{document}